\documentclass[11pt]{article}

\usepackage{amssymb,amsthm,amsmath}
\usepackage{mathrsfs}
\usepackage[dvipdfmx]{graphicx}
\usepackage[FIGTOPCAP,nooneline]{subfigure}
\usepackage{geometry}

\usepackage{color}

\makeatletter
    
    \@addtoreset{equation}{section}
\makeatother

\theoremstyle{definition}
\newtheorem{theorem}{Theorem}[section]

\newtheorem{definition}[theorem]{Definition}

\newtheorem{lemma}[theorem]{Lemma}

\usepackage[T1]{fontenc}
\usepackage[utf8]{inputenc}
\usepackage{authblk}

\title{Convergence of filtered weak solutions to the 2D Euler equations with vortex sheet initial data}

\author{Takeshi Gotoda \footnote{Research Institute for Electronic Science, Hokkaido University, Kita 12 Nishi 7, Kita-ku, Sapporo,  Hokkaido, JAPAN  E-mail : gotoda@es.hokudai.ac.jp} }

\date{}

\begin{document}


\maketitle

\begin{abstract}
We study weak solutions of the two-dimensional (2D) filtered Euler equations whose vorticity is a finite Radon measure and velocity has locally finite kinetic energy, which is called the vortex sheet solution. The filtered Euler equations are a regularized model based on a spatial filtering to the Euler equations. The 2D filtered Euler equations have a unique global weak solution for measure valued initial vorticity, while the 2D Euler equations require initial vorticity to be in the vortex sheet class with a distinguished sign for the existence of global solutions. In this paper, we prove that vortex sheet solutions of the 2D filtered Euler equations converge to those of the 2D Euler equations in the limit of the filtering parameter provided that initial vortex sheet has a distinguished sign. We also show that a simple application of our proof yields the convergence of the vortex blob method for vortex sheet solutions. Moreover, we make it clear what kind of condition should be imposed on the spatial filter to show the convergence results and, according to the condition, these results are applicable to well-known regularized models like the Euler-$\alpha$ model and the vortex blob model.

\end{abstract}


\section{Introduction}

We are concerned with solutions of incompressible and inviscid equations with {\it vortex sheet initial data}. The motion of inviscid incompressible 2D flows is described by the 2D Euler equations:
\begin{equation}
\partial_t \boldsymbol{v} + (\boldsymbol{v} \cdot \nabla) \boldsymbol{v} + \nabla p = 0, \qquad \operatorname{div} \boldsymbol{v} = 0, \qquad \boldsymbol{v}(\boldsymbol{x}, 0) = \boldsymbol{v}_0(\boldsymbol{x}),  \label{Euler}
\end{equation}
where $\boldsymbol{v}=\boldsymbol{v}(\boldsymbol{x}, t) =  ( v_1(\boldsymbol{x}, t), v_2(\boldsymbol{x}, t))$ is the fluid velocity field, $p=p(\boldsymbol{x}, t)$ is the scalar pressure, and $\boldsymbol{v}_0$ is the given initial velocity.  An incompressible velocity field $\boldsymbol{v}_0$ is called vortex sheet initial data provided that its vorticity $q_0 = \operatorname{curl} \boldsymbol{v}_0 = \partial_{x_1} v_2 - \partial_{x_2} v_1 $ belongs to the space of finite Radon measures on $\mathbb{R}^2$, denoted by $\mathcal{M}(\mathbb{R}^2)$, and $\boldsymbol{v}_0$ has locally finite kinetic energy, i.e.,
\begin{equation*}
\int_{|\boldsymbol{x}| \leq R} |\boldsymbol{v}_0|^2 d\boldsymbol{x} < C(R)
\end{equation*}
for any $R > 0$. To consider flows such as vortex sheet, it is required to consider a weak form of the Euler equations. The classical weak formulation leads the following definition of a weak solution in weak velocity form, see \cite{Diperna}.

\begin{definition}
An incompressible velocity field $\boldsymbol{v} \in L^\infty([0, T]; L^2_{\mathrm{loc}}(\mathbb{R}^2))$ is a weak solution of (\ref{Euler}) provided that
\begin{itemize}
\item[(i)] for any $\Psi =(\psi_1, \psi_2)$ with $\psi_1$, $\psi_2 \in C_c^\infty([0,T]\times\mathbb{R}^2)$ and $\operatorname{div} \Psi = 0$,
\begin{equation*}
\int_0^T \int_{\mathbb{R}^2} \left( \partial_t \Psi \cdot \boldsymbol{v} + \nabla \Psi : \boldsymbol{v} \otimes \boldsymbol{v} \right) d\boldsymbol{x} dt = 0,
\end{equation*}
\item[(ii)] $\boldsymbol{v}$ belongs to $\mathrm{Lip}([0,T];H_{\mathrm{loc}}^{-L}(\mathbb{R}^2))$ for some $L > 0$ and $\boldsymbol{v}(0) = \boldsymbol{v}_0$ in $H_{\mathrm{loc}}^{-L}(\mathbb{R}^2)$,
\end{itemize}
where $\boldsymbol{v} \otimes \boldsymbol{v} = (v_i v_j)$, $\nabla \Psi = (\partial_j \psi_i)$ and $A : B = \sum_{i,j} a_{ij} b_{ij}$. $H^s(\mathbb{R}^2)$ denotes the Sobolev space and the localized Sobolev space $H_{\mathrm{loc}}^s(\mathbb{R}^2)$ is the space of distributions $f$ such that $\rho f \in H^s(\mathbb{R}^2)$ for any $\rho \in C_c^\infty(\mathbb{R}^2)$.
\label{def-sol-velocity}
\end{definition}

Taking the curl of (\ref{Euler}) and setting the vorticity $q = \operatorname{curl} \boldsymbol{v}$, we obtain the transport equation for $\omega$: 
\begin{equation}
\partial_t q + (\boldsymbol{v} \cdot \nabla) q = 0, \qquad q(\boldsymbol{x}, 0) = q_0(\boldsymbol{x}).  \label{vEuler}
\end{equation}
Then, the velocity $\boldsymbol{v}$ is recovered from the vorticity $\omega$ via the Biot-Savart law:
\begin{equation}
\boldsymbol{v}(\boldsymbol{x}) = \left( \boldsymbol{K} \ast q \right)(\boldsymbol{x}) = \int_{\mathbb{R}^2} \boldsymbol{K}(\boldsymbol{x} - \boldsymbol{y}) q(\boldsymbol{y}) d \boldsymbol{y}, \label{biot}
\end{equation}
in which $\boldsymbol{K}$ is a singular integral kernel defined by 
\begin{equation*}
\boldsymbol{K}(\boldsymbol{x}) = \nabla^\perp G(\boldsymbol{x}) = \frac{1}{2 \pi} \frac{\boldsymbol{x}^\perp}{\left| \boldsymbol{x} \right|^2}, \qquad G(\boldsymbol{x}) = \frac{1}{2 \pi} \log{|\boldsymbol{x}|}
\end{equation*}
with $\nabla^\perp = (- \partial_{x_2}, \partial_{x_1})$ and $\boldsymbol{x}^\perp = (- x_2, x_1)$. The function $G$ is a fundamental solution to the 2D Laplacian. Since the weak formulation of (\ref{vEuler}) is described by
\begin{equation*}
\int_0^T \int_{\mathbb{R}^2} q \left( \partial_t \psi + \nabla \psi \cdot \boldsymbol{v} \right) d\boldsymbol{x} dt = 0,
\end{equation*}
substituting (\ref{biot}), the weak vorticity form of the 2D Euler equations is obtained as follows, see \cite{Schochet}.

\begin{definition}
A vorticity $q \in L^\infty( [0, T] ; \mathcal{M}(\mathbb{R}^2)\cap H_{\mathrm{loc}}^{-1}(\mathbb{R}^2))$ is a weak solution of (\ref{vEuler}) and (\ref{biot}) provided that 
\begin{itemize}
\item[(i)] for any $\psi\in C_c^\infty(\mathbb{R}^2\times(0,T))$,
\begin{align*}
0 &= \int_0^T \int_{\mathbb{R}^2} \partial_t \psi(\boldsymbol{x},t) q(\boldsymbol{x},t) d\boldsymbol{x} dt \\
&\quad + \int_0^T \int_{\mathbb{R}^2} \int_{\mathbb{R}^2} H_{\psi}(\boldsymbol{x},\boldsymbol{y},t) q(\boldsymbol{y},t) q(\boldsymbol{x},t) d\boldsymbol{y} d\boldsymbol{x} dt \\
&\equiv W_L(q;\psi) + W_{NL}(q;\psi),
\end{align*}
where 
\begin{equation*}
H_{\psi}(\boldsymbol{x},\boldsymbol{y},t) = \frac{1}{2} \boldsymbol{K}(\boldsymbol{x} - \boldsymbol{y}) \cdot \left( \nabla \psi(\boldsymbol{x},t) - \nabla \psi(\boldsymbol{y},t) \right),
\end{equation*}
\item[(ii)] $q \in \mathrm{Lip}([0,T];H_{\mathrm{loc}}^{-(L+1)}(\mathbb{R}^2))$ for some $L > 0$ and $q(0) = q_0$ in $H_{\mathrm{loc}}^{-(L+1)}(\mathbb{R}^2)$.
\end{itemize}
\label{def-sol-vorticity}
\end{definition}
It is important to remark that, for any fixed $t \in (0,T)$, $H_{\psi}(\boldsymbol{x}, \boldsymbol{y}, t)$ is bounded on $\mathbb{R}^2 \times \mathbb{R}^2$, continuous outside the diagonal $\boldsymbol{x} = \boldsymbol{y}$, and tends to zero at infinity \cite{Delort}. The global existence of a weak solution of (\ref{vEuler}) and (\ref{biot}) has been established for $q_0 \in L^1(\mathbb{R}^2) \cap L^p(\mathbb{R}^2)$ with $1 < p \leq \infty$ and the uniqueness holds for the case of $p = \infty$\cite{Marchioro,Yudovich,Diperna}. By the Delort's result \cite{Delort}, the existence theorem is extended to the case of vortex sheet initial data $q_0 \in \mathcal{M}(\mathbb{R}^2)\cap H_{\mathrm{loc}}^{-1}(\mathbb{R}^2)$ with a distinguished sign, in which the solution $q \in L^\infty( [0, T] ; \mathcal{M}(\mathbb{R}^2)\cap H_{\mathrm{loc}}^{-1}(\mathbb{R}^2))$ satisfies the 2D Euler equations in the sense of Definition~\ref{def-sol-velocity}, see \cite{Schochet}. 


Another approach for analyzing the dynamics of vortex sheet employs the Birkhoff-Rott equation \cite{Birkhoff,Rott} described by
\begin{equation}
\frac{\partial \overline{z}}{\partial t} (\Gamma, t) = \frac{1}{2 \pi i} \mathrm{p.v.} \int_{-\infty}^\infty \frac{d \Gamma'}{z(\Gamma, t) - z(\Gamma', t)}, \label{BR}
\end{equation}
where $z(\Gamma, t) = x_1(\Gamma, t) + i x_2(\Gamma ,t)$ is a complex curve of the vortex sheet and $\Gamma$ represents the circulation. The Birkhoff-Rott equation is is formally derived from the 2D Euler equations with the assumption that the vorticity is concentrated on a curve for any fixed time. Although, for sufficiently {\it regular} vortex sheets, the equivalence between descriptions of the Birkhoff-Rott equation and weak forms of the Euler equations has been established in \cite{Lopes}, unlike weak forms of the 2D Euler equations, its initial-value problem seems to be ill-posed in the sense of Hadamard due to the Kelvin-Helmholtz instability. Indeed, the local existence of an analytic solution of (\ref{BR}) was proven in \cite{Caflisch,Sulem}, but asymptotic analysis and numerical computation suggest that vortex sheets develop a singularity in a finite time even for smooth initial curves \cite{Krasny(a),Meiron,Moore}. To study the behavior of a vortex sheet beyond the critical time, an inviscid regularization called the vortex blob method was introduced in numerical simulations \cite{Anderson,Chorin,Krasny} and it has been applied to a variety of numerical solutions of vortex sheets \cite{Krasny(b),Leonard,Nitsche}. These numerical simulations indicate that vortex sheets roll up into a spirals past the critical time and that infinite length spirals appear in the limit of the regularization parameter. For the vortex blob method, it has been shown in \cite{Liu} that approximated solutions converge to weak solutions of the 2D Euler equations with vortex sheet initial data.

The present paper treats a vortex sheet problem in terms of weak solutions of the 2D Euler equations. We consider the filtered Euler equations, which is derived by applying a spatial filtering to the Euler equations. Especially, we show the convergence of vortex sheet solutions of the 2D filtered Euler equations or their approximation by point-vortices converge to those to the 2D Euler equations. As mentioned later, some similar results have been shown for the vortex blob method or the Euler-$\alpha$ equations that is also a example of the filtered model. Our main purpose is to deal with the problems in a unified manner by analyzing the filtered Euler equations and make it clear what kind of condition for the filter is essential to show the convergence.

\section{The filtered Euler equations}

The 2D filtered Euler equations are an inviscid regularization of the 2D Euler equations and given by
\begin{equation}
\partial_t \boldsymbol{v}^\varepsilon + (\boldsymbol{u}^\varepsilon \cdot \nabla) \boldsymbol{v}^\varepsilon - (\nabla \boldsymbol{v}^\varepsilon)^T \cdot \boldsymbol{u}^\varepsilon + \nabla \Pi = 0, \label{REE}
\end{equation}
where $\Pi$ is the generalized pressure and $\boldsymbol{u}^\varepsilon$ is a spatially filtered velocity defined by
\begin{equation}
\boldsymbol{u}^\varepsilon(\boldsymbol{x}) = \left( h^\varepsilon \ast \boldsymbol{v}^\varepsilon \right) (\boldsymbol{x}) = \int_{\mathbb{R}^2} h^\varepsilon\left( \boldsymbol{x} - \boldsymbol{y} \right) \boldsymbol{v}^\varepsilon(\boldsymbol{y}) d\boldsymbol{y}, \quad h^\varepsilon(\boldsymbol{x}) = \frac{1}{\varepsilon^2} h \left( \frac{\boldsymbol{x}}{\varepsilon} \right) \label{Rvelo}
\end{equation}
with a scalar valued function $h \in L^1(\mathbb{R}^2)$. Refer to \cite{Foias,Holm(c)} for the derivation of the filtered Euler equations on the basis of (\ref{Rvelo}). In this paper, we call $h$ the {\it smoothing function} and, for simplicity, assume that $h$ is a positive and radial function and may have a singularity at the origin. Taking the curl of (\ref{REE}) with the incompressible condition, we obtain the transport equation for $q^\varepsilon = \operatorname{curl} \boldsymbol{v}^\varepsilon$ convected by $\boldsymbol{u}^\varepsilon$,
\begin{equation}
\partial_t q^\varepsilon + (\boldsymbol{u}^\varepsilon \cdot \nabla) q^\varepsilon = 0. \label{VFEE}
\end{equation}
Note that $\operatorname{div} \boldsymbol{u}^\varepsilon = 0$ and $\omega^\varepsilon = \operatorname{curl} \boldsymbol{u}^\varepsilon$ satisfies $\omega^\varepsilon = h^\varepsilon \ast q^\varepsilon$ when the convolution commutes with the differential operator. Since the Biot-Savart law gives $\boldsymbol{u}^\varepsilon = \boldsymbol{K} \ast \omega^\varepsilon$, the filtered velocity $\boldsymbol{u}^\varepsilon$ is recovered from the vorticity $q^\varepsilon$ as follows.
\begin{equation}
\boldsymbol{u}^\varepsilon = \boldsymbol{K}^\varepsilon \ast q, \qquad \boldsymbol{K}^\varepsilon = \boldsymbol{K} \ast h^\varepsilon. \label{F-Biot-Savart}
\end{equation}
The Lagrangian flow map $\boldsymbol{\eta}^\varepsilon$ associated with $\boldsymbol{u}^\varepsilon$ is given by
\begin{equation*}
\partial_t \boldsymbol{\eta}^\varepsilon(\boldsymbol{x}, t) = \boldsymbol{u}^\varepsilon \left( \boldsymbol{\eta}^\varepsilon(\boldsymbol{x}, t), t \right), \qquad  \boldsymbol{\eta}^\varepsilon(\boldsymbol{x}, 0) = \boldsymbol{x}. \label{flow-map}
\end{equation*}
We remark that $\boldsymbol{K}^\varepsilon$ satisfies $\boldsymbol{K}^\varepsilon = \nabla^\perp G^\varepsilon$, where $G^\varepsilon = G \ast h^\varepsilon$ is a solution to the 2D Poisson equation, $\Delta G^\varepsilon = h^\varepsilon$. Since $h$ is a radial function, $G^\varepsilon$ is also radial, i.e. $G^\varepsilon(\boldsymbol{x}) = G_r^\varepsilon(\vert\boldsymbol{x}\vert)$, and we have
\begin{equation}
\boldsymbol{K}^\varepsilon(\boldsymbol{x}) = \boldsymbol{K}(\boldsymbol{x}) P_K \left( \frac{\vert\boldsymbol{x}\vert}{\varepsilon} \right), \qquad P_K(r) = 2 \pi r \frac{\mbox{d}G_r^{\varepsilon=1}}{\mbox{d}r} (r). \label{K^eps-KP}
\end{equation}
Here, $P_K(r)$ is a monotonically increasing function satisfying $P_K(0) =0$ and $P_k(r) \rightarrow 1$ as $r \rightarrow \infty$, see \cite{G.}.

It is noteworthy that, taking a specific filter $h$, we obtain the well-known regularizations: the Euler-$\alpha$ model and the vortex blob model. Although both models are originally derived in different backgrounds, they are generalized in a unified manner by the spatial filtering \cite{Foias,Holm(c)}. The Euler-$\alpha$ model was first introduced in \cite{Holm(a),Holm(b)} on the basis of the Euler-Poincar\'{e} variational framework or the Lagrangian averaging to the Euler equations \cite{Marsden}. In the viewpoint of physical significance, the viscous extension of the Euler-$\alpha$ equations, the Navier-Stokes-$\alpha$ equations, are utilized as a turbulent model \cite{Chen(a),Chen,Foias,Foias(a),Lunasin,Mohseni}. In the 2D Euler-$\alpha$ model, the smoothing function $h^\alpha$ is given by a fundamental solution for the operator $1 - \alpha^2 \Delta$, that is, 
\begin{equation*}
h^\alpha(\boldsymbol{x}) = \frac{1}{\alpha^2} h \left( \frac{\boldsymbol{x}}{\alpha} \right), \qquad h(\boldsymbol{x}) = \frac{1}{2\pi} K_0 \left( |\boldsymbol{x}| \right),
\end{equation*}
where $K_0$ denotes the modified Bessel function of the second kind. Then, since we have $\boldsymbol{v}^\alpha = (1 - \alpha^2 \Delta) \boldsymbol{u}^\alpha $, the 2D filtered Euler equations for $\boldsymbol{u}^\alpha$ is described by
\begin{equation*}
(1 - \alpha^2 \Delta) \partial_t \boldsymbol{u}^\alpha + \boldsymbol{u}^\alpha \cdot \nabla (1 - \alpha^2 \Delta) \boldsymbol{u}^\alpha + (\nabla \boldsymbol{u}^\alpha)^T \cdot (1 - \alpha^2 \Delta) \boldsymbol{u}^\alpha = \nabla \Pi,
\end{equation*}
which are known as the Euler-$\alpha$ equations. Note that there exists a unique global weak solution for initial vorticity $q_0 \in \mathcal{M}(\mathbb{R}^2)$ and a weak solution for $q_0 \in L^\infty(\mathbb{R}^2)\cap L^1(\mathbb{R}^2)$ converges to that of the 2D Euler equations in the $\alpha \rightarrow 0$ limit \cite{Oliver}. Moreover, the convergence to the 2D Euler equations holds for vortex sheet initial data, that is, $q_0 \in \mathcal{M}(\mathbb{R}^2)$ with a distinguished sign and $\boldsymbol{v}_0 \in L_{\mathrm{loc}}^2(\mathbb{R}^2)$ \cite{Bardos}. 

On the other hand, the vortex blob method was introduced for a numerical scheme to compute the vortex sheets' evolutions in incompressible inviscid flows. In that method, the initial vorticity is approximated by a Dirac delta functions and the Biot-Savart formula is regularized by (\ref{F-Biot-Savart}) with the smoothing function $h^\sigma$ defined by
\begin{equation*}
h^\sigma(\boldsymbol{x}) = \frac{1}{\sigma^2} h \left( \frac{\boldsymbol{x}}{\sigma} \right), \qquad h(\boldsymbol{x}) = \frac{1}{\pi (|\boldsymbol{x}|^2 + 1)^2}.
\end{equation*}
Then, the filtered velocity field induced by a point-vortex at $\boldsymbol{x}_0$ is given by
\begin{equation*}
\boldsymbol{u}^\sigma(\boldsymbol{x}) = \frac{1}{2 \pi} \frac{(\boldsymbol{x}- \boldsymbol{x}_0)^\perp}{|\boldsymbol{x} - \boldsymbol{x}_0|^2 + \sigma^2}.
\end{equation*}
Unlike the Euler-$\alpha$ equations, the governing equation for $\boldsymbol{u}^\sigma$ is not described explicitly. The convergence of the vortex blob method in the $\sigma \rightarrow 0$ limit has been shown in \cite{Liu}, in which approximate solutions by the vortex blob method converge to weak solutions of the 2D Euler equations provided that $q_0 \in \mathcal{M}(\mathbb{R}^2)$ with a distinguished sign and $\boldsymbol{v}_0 \in L_{\mathrm{loc}}^2(\mathbb{R}^2)$.

From the above, the filtered Euler models are regarded as significant models for understanding physical phenomena in fluid mechanics. In the preceding study \cite{G.}, the global well-posedness for the 2D filtered models is considered in a unified manner and it has been shown that the 2D filtered Euler equations (\ref{VFEE}) with (\ref{F-Biot-Savart}) have a unique global weak solution with initial vorticity in $\mathcal{M}(\mathbb{R}^2)$. We remark that it is not necessary to assume that $q_0$ belongs to $H^{-1}_{\mathrm{loc}}(\mathbb{R}^2)$ and has a single sign, which implies that a unique global weak solution exits for point vortex initial vorticity, and this is a theoretical advantage of the 2D filtered Euler equations. As for the vortex sheet problem, since we have a unique Lagrangian flow map $\boldsymbol{\eta}^\varepsilon$ satisfying (\ref{flow-map}), the vorticity is concentrated on a curve for any time if initial vorticity is supported on a curve. Then, following the argument in \cite{Bardos}, the evolution of vortex sheets is described by the {\it filtered Birkhoff-Rot equation},
\begin{equation*}
\frac{\partial \boldsymbol{x}}{\partial t} (\Gamma, t) = \int_{-\infty}^\infty \boldsymbol{K}^\varepsilon(\boldsymbol{x}(\Gamma, t) - \boldsymbol{x}(\Gamma', t)) d \Gamma', \label{FBR}
\end{equation*}
which is equivalent to the weak formulation of the 2D filtered Euler equations, owing to the filtering regularization.

We show fundamental properties of the 2D filtered Euler equations. A weak solution $q^\varepsilon\in C([0,T];\mathcal{M}(\mathbb{R}^2))$ constructed in \cite{G.} satisfies the 2D filtered Euler equations in the sense that
\begin{equation*}
\int_0^T \int_{\mathbb{R}^2} \left( \partial_t \psi(\boldsymbol{x},t)  + \boldsymbol{u}^\varepsilon(\boldsymbol{x},t) \cdot \nabla \psi(\boldsymbol{x},t) \right) q^\varepsilon(\boldsymbol{x},t) d\boldsymbol{x} dt = 0  \label{WVFEE}
\end{equation*}
for any $\psi\in C_c^\infty(\mathbb{R}^2\times(0,T))$, which is equivalent to
\begin{align*}
0 = &\int_0^T \int_{\mathbb{R}^2} \partial_t \psi(\boldsymbol{x},t) q^\varepsilon(\boldsymbol{x},t) d\boldsymbol{x} dt \\
&\quad + \int_0^T \int_{\mathbb{R}^2} \int_{\mathbb{R}^2} H^\varepsilon_{\psi}(\boldsymbol{x},\boldsymbol{y},t) q^\varepsilon(\boldsymbol{y},t) q^\varepsilon(\boldsymbol{x},t) d\boldsymbol{y} d\boldsymbol{x} dt \\
&\equiv W_L(q^\varepsilon; \psi) + W_{NL}^\varepsilon(q^\varepsilon; \psi),
\end{align*}
where 
\begin{align*}
H^\varepsilon_{\psi}(\boldsymbol{x},\boldsymbol{y},t) &= \frac{1}{2} \boldsymbol{K}^\varepsilon(\boldsymbol{x} - \boldsymbol{y}) \cdot \left( \nabla \psi(\boldsymbol{x},t) - \nabla \psi(\boldsymbol{y},t) \right).
\end{align*}
Owing to $q^\varepsilon \in C([0,T];\mathcal{M}(\mathbb{R}^2))$, the initial condition $q^\varepsilon(0) = q_0$ in $\mathcal{M}(\mathbb{R}^2)$ makes sense and $q^\varepsilon$ is expressed by $q^\varepsilon(\boldsymbol{x}, t) = q_0 \left( \boldsymbol{\eta}^\varepsilon(\boldsymbol{x}, - t) \right)$, which yields $\| q^\varepsilon(\cdot, t) \|_{\mathcal{M}} = \| q_0 \|_{\mathcal{M}}$. The filtered vorticity $\omega^\varepsilon$ also belongs to $C([0,T];\mathcal{M}(\mathbb{R}^2))$, since it follows that 
\begin{align*}
\left| \int_{\mathbb{R}^2} \psi(\boldsymbol{x}) \omega^\varepsilon(\boldsymbol{x}) d\boldsymbol{x} \right| \leq \| \psi \ast h^\varepsilon \|_{L^\infty} \| q^\varepsilon \|_{\mathcal{M}} \leq \| \psi \|_{L^\infty} \| q^\varepsilon \|_{\mathcal{M}}
\end{align*}
for any $\psi \in C_0(\mathbb{R}^2)$, the space of continuous functions vanishing at infinity, and $\| \omega^\varepsilon(\cdot, t) \|_{\mathcal{M}} \leq \| q^\varepsilon(\cdot, t) \|_{\mathcal{M}} = \| q_0 \|_{\mathcal{M}}$. In this paper, we omit the domain in the norm when it is the entire space $\mathbb{R}^2$. We also remark that $H^\varepsilon_{\psi}(\boldsymbol{x},\boldsymbol{y},t)$ is continuous on $\mathbb{R}^2 \times \mathbb{R}^2$ and tends to zero at infinity for any $t \in (0,T)$. Indeed, $\boldsymbol{K}^\varepsilon$ is a quasi-Lipschitz continuous function in $C_0(\mathbb{R}^2)$ under the assumption of the global well-posedness in \cite{G.}, which enables us to define the filtered Biot-Savart formula (\ref{F-Biot-Savart}) with the vorticity in $\mathcal{M}(\mathbb{R}^2)$.

\section{Main results}

In this paper, we investigate the $\varepsilon \rightarrow 0$ limit of solutions of the 2D filtered Euler equations, (\ref{VFEE}) and (\ref{F-Biot-Savart}), with vortex sheet initial data. For the 2D Euler-$\alpha$ equations, it has been shown in \cite{Bardos} that vortex sheet solutions with a distinguished sign converge to those of the 2D Euler equations in the $\alpha \rightarrow 0$ limit. Our question is whether the convergence result obtained in the Euler-$\alpha$ model holds for the filtered Euler equations in the limit of $\varepsilon \rightarrow 0$. This is not an obvious extension of the result in \cite{Bardos}. Indeed, according to their proof, the uniform decay in small disks of the filtered vorticity $\omega^\alpha$,
\begin{equation*}
\sup_{0\leq t \leq T, \boldsymbol{x}_0 \in\mathbb{R}^2 } \int_{|\boldsymbol{x} - \boldsymbol{x}_0|\leq r} \omega^\alpha(\boldsymbol{x}, t) d\boldsymbol{x} < c(T) \left[ \log{\left(\frac{1}{r}\right)} \right]^{-1/2},
\end{equation*}
plays essential roles to show the convergence theorem, and it is derived by using the feature of the modified Bessel function appearing in the smoothing function $h^\alpha$. In our proof, in order to treat the filtered models in a unified manner, we follow the ideas in \cite{Majda,Majda-1}, that is, derive the uniform decay for $q^\varepsilon$ by using the pseudo-energy of the 2D filtered Euler equations. Our another concern is what is the essential property of the smoothing function for the convergence. We clarify the functional class of smoothing functions that guarantees the convergence of filtered vortex sheet solutions to weak solutions of the Euler equations regardless of spatial filters.

To state main theorems, we introduce a weighting function by $w_{\alpha}(\boldsymbol{x}) \equiv  |\boldsymbol{x}|^\alpha$ for $\boldsymbol{x} \in \mathbb{R}^2$, where $\chi_{A}(\boldsymbol{x})$ is an indicator function.

\begin{theorem}
{\it We assume that the 2D filtered Euler equations with the smoothing function $h$ have a unique global weak solution for $q_0 \in \mathcal{M}(\mathbb{R}^2)$. Suppose that $h \in C_0(\mathbb{R}^2 \setminus \{ 0 \}) \cap L^1(\mathbb{R}^2)$ is a positive and radial function satisfying
\begin{equation}
w_1 h \in L^1(\mathbb{R}^2), \qquad w_3 h \in L^\infty(\mathbb{R}^2). \label{h-condi-1}
\end{equation}
Let $(\boldsymbol{u}^\varepsilon, q^\varepsilon)$ be a solution of the 2D filtered Euler equations for $q_0 \in \mathcal{M}(\mathbb{R}^2) \cap H^{-1}_{\mathrm{loc}}(\mathbb{R}^2)$ with a distinguished sign and compact support. Then, for any $T > 0$, there exist subsequences $\{\boldsymbol{u}^{\varepsilon_j}\}$, $\{ q^{\varepsilon_j} \}$ and their limits $\boldsymbol{u} \in L^2_{\mathrm{loc}}(\mathbb{R}^2 \times [0,T])$, $q = \operatorname{curl} \boldsymbol{u} \in L^\infty([0,T];\mathcal{M}(\mathbb{R}^2))$ such that 
\begin{align*}
& q^{\varepsilon_j} \overset{\ast}{\rightharpoonup} q  \quad \mathrm{weakly}\ast \ \mathrm{in} \ L^\infty([0,T];\mathcal{M}(\mathbb{R}^2)), \\
& \boldsymbol{u}^{\varepsilon_j} \rightharpoonup \boldsymbol{u} \quad \mathrm{weakly} \ \mathrm{in} \ L^2_{\mathrm{loc}}(\mathbb{R}^2 \times [0,T])
\end{align*}
and $(\boldsymbol{u}, q)$ are a weak solution of the 2D Euler equations with initial vorticity $q_0$ in the sense of Definition~\ref{def-sol-velocity} and Definition~\ref{def-sol-vorticity} respectively.
}\label{convergence-thm1}
\end{theorem}

The proofs of Theorem~\ref{convergence-thm1} are based on the Delort's proof in \cite{Delort} and its simple proof given by \cite{Majda}, in which a uniform estimate for the decay of the local circulation plays crucial roles.

In addition, we consider a point-vortex approximation of the vortex sheet initial vorticity $q_0 \in \mathcal{M}(\mathbb{R}^2) \cap H^{-1}_{\mathrm{loc}}(\mathbb{R}^2)$. We begin with introducing the filtered point-vortex system \cite{G.2,Liu}. Let $\{ S_n \}$ be non-overlapping squares covering the support of $q_0$ with side lengths $\eta$ and centered at $\boldsymbol{c}_n = \boldsymbol{j}_n \eta$, where $\boldsymbol{j}_n \in\mathbb{Z}^2$. The initial circulation of $q_0$ in $S_n$ is defined by
\begin{equation*}
\Gamma_n = \int_{S_n} q_0(\boldsymbol{x}) d\boldsymbol{x}.
\end{equation*}
In the vortex method, the initial vortex sheet $q_0$ is approximated by point vortices as follows.
\begin{equation}
\widehat{q}_0(\boldsymbol{x}) = \sum_{n \in \mathbb{N}} \Gamma_n \delta(\boldsymbol{x} - \boldsymbol{c}_n). \label{init-pv}
\end{equation}
As discussed in \cite{G.2}, the solution of the 2D filtered Euler equations for initial vorticity (\ref{init-pv}) is expressed by
\begin{equation}
q^\varepsilon(\boldsymbol{x}, t) = \sum_{n \in \mathbb{N}} \Gamma_n \delta(\boldsymbol{x} - \boldsymbol{x}^\varepsilon_n(t)), \label{sol-pv}
\end{equation}
in which the supports $\boldsymbol{x}^\varepsilon_n(t)$ are governed by the following Hamiltonian system,
\begin{equation}
\frac{\mbox{d}}{\mbox{d}t} \boldsymbol{x}^\varepsilon_n(t) = \sum_{m \neq n} \Gamma_m \boldsymbol{K}^\varepsilon \left( \boldsymbol{x}^\varepsilon_n(t) - \boldsymbol{x}^\varepsilon_m(t) \right), \qquad \boldsymbol{x}^\varepsilon_n(0) = \boldsymbol{c}_n. \label{FPV}
\end{equation}
Note that the filtered velocity field of (\ref{sol-pv}) is given by
\begin{equation}
\boldsymbol{u}^\varepsilon(\boldsymbol{x}, t) = \sum_{n \in \mathbb{N}} \Gamma_n \boldsymbol{K}^\varepsilon \left( \boldsymbol{x} - \boldsymbol{x}^\varepsilon_n(t) \right). \label{velo-pv}
\end{equation}
Regarding the point-portex method and the vortex blob method, the convergences of approximated solutions have been shown for smooth solutions to the Euler equations, see \cite{Beale,Goodman,Hald}. The convergence result is extended to the case of vortex sheets: it has been shown in \cite{Liu} that vortex sheet solutions with a distinguished sign converge to weak solutions constructed in \cite{Delort}. Similar to the preceding result in \cite{Liu}, we show that approximated vortex sheet solutions by the filtered point-vortex system converge to weak solutions of the 2D Euler equations by applying the proof of Theorem~\ref{convergence-thm1}. The following result seems to be the same as that shown in \cite{Liu}, but the assumption for the smoothing function is different. We remark that our result allows the smoothing function $h$ to have a singularity at the origin so that the theorem can be applied to the Euler-$\alpha$ model, while the preceding results suppose the boundedness of $h$. This is an important extension since the Euler-$\alpha$ model is closer to the unregularized flow in the sense of the smoothing function \cite{Holm(c)}.

\begin{theorem}
{\it  Suppose that $h$ satisfies the assumption of Theorem~\ref{convergence-thm1} and $q_0 \in \mathcal{M}(\mathbb{R}^2) \cap H^{-1}_{\mathrm{loc}}(\mathbb{R}^2)$ has a distinguished sign and compact support. Let $(\boldsymbol{u}^\varepsilon, q^\varepsilon)$ be the approximate solution of the 2D filtered Euler equations generated by the vortex method as described above. Then, for any $T > 0$, there exist subsequences $\{\boldsymbol{u}^{\varepsilon_j}\}$, $\{ q^{\varepsilon_j} \}$ and their limits $\boldsymbol{u} \in L^2_{\mathrm{loc}}(\mathbb{R}^2 \times [0,T])$, $q = \operatorname{curl} \boldsymbol{u} \in L^\infty([0,T];\mathcal{M}(\mathbb{R}^2))$ such that 
\begin{align*}
& q^{\varepsilon_j} \overset{\ast}{\rightharpoonup} q  \quad \mathrm{weakly}\ast \ \mathrm{in} \ L^\infty([0,T];\mathcal{M}(\mathbb{R}^2)), \\
& \boldsymbol{u}^{\varepsilon_j} \rightharpoonup \boldsymbol{u} \quad \mathrm{weakly} \ \mathrm{in} \ L^2_{\mathrm{loc}}(\mathbb{R}^2 \times [0,T])
\end{align*}
and $(\boldsymbol{u}, q)$ are a weak solution of the 2D Euler equations with initial vorticity $q_0$ in the sense of Definition~\ref{def-sol-velocity} and Definition~\ref{def-sol-vorticity} respectively, provided that the grid size $\eta$ satisfies $\eta \leq c \varepsilon$ with some constant $c >0$.
}\label{convergence-thm2}
\end{theorem}

According to Theorem~\ref{convergence-thm2}, the condition (\ref{h-condi-1}) describes the decay rate of $h$ for the convergence. On the other hand, the rate of the singularity at the origin is estimated by the solvability of the 2D filtered Euler equations, which is discussed in \cite{G.}.

\label{Main}

\section{Approximate-solution sequence}
\subsection{Vorticity maximal function}
In this section, we introduce an approximate-solution sequence of the 2D filtered Euler equations by smoothing initial data and show a decay property of the circulation. For a given vorticity $q$, the circulation inside a circle with center $\boldsymbol{x}_0 \in \mathbb{R}^2$ and radius $r$ is given by 
\begin{equation*}
\int_{|\boldsymbol{x} - \boldsymbol{x}_0|\leq r} q(\boldsymbol{x}, t) d\boldsymbol{x}.
\end{equation*}
Following \cite{Diperna}, we call the function,
\begin{equation*}
M(r, q) \equiv \sup_{0\leq t \leq T, \boldsymbol{x}_0 \in\mathbb{R}^2 } \int_{|\boldsymbol{x} - \boldsymbol{x}_0|\leq r} q(\boldsymbol{x}, t) d\boldsymbol{x},
\end{equation*}
the {\it vorticity maximal function} for $q$. We now show that the vorticity maximal function for the approximate-solution sequence of the 2D filtered Euler equations uniformly decays with a sufficiently rapid rate as $r \rightarrow \infty$. 

Let $\boldsymbol{v}_0 \in L^2_{\mathrm{loc}}(\mathbb{R}^2)$ and $q_0 = \operatorname{curl} \boldsymbol{v}_0 \in \mathcal{M}(\mathbb{R}^2) \cap H^{-1}_{\mathrm{loc}}(\mathbb{R}^2)$ be vortex sheet initial data and we assume $q_0 \geq 0$ without loss of generality. We define smoothed initial data by
\begin{equation}
\boldsymbol{v}_0^\delta = \phi^\delta \ast \boldsymbol{v}_0, \qquad q_0^{\delta} = \operatorname{curl} \boldsymbol{v}_0^\delta, \label{smooth_ID}
\end{equation} 
where $\phi^\delta(\boldsymbol{x}) = \delta^{-2} \phi(\boldsymbol{x} / \delta)$ is a standard mollifier with $\phi \in C_c^\infty(\mathbb{R}^2)$ satisfying $\phi \geq 0$, supp $\phi \subset B_1$ and $\int_{\mathbb{R}^2} \phi(\boldsymbol{x}) d\boldsymbol{x} = 1$. Here, we used the notation $B_r = \{ \boldsymbol{x}\in \mathbb{R}^2 \ | \ |\boldsymbol{x}| < r \}$. Then, we easily find $\boldsymbol{v}_0^\delta \in C^\infty(\mathbb{R}^2)$ and $q_0^\delta \in C_c^\infty(\mathbb{R}^2)$ with supp $q_0^\delta \subset B_{\overline{R}}$ for some $\overline{R} > 0$. In addition, $\boldsymbol{v}_0^\delta$ is uniformly bounded in $L^2_{\mathrm{loc}}(\mathbb{R}^2)$, since we have 
\begin{align*}
\| \boldsymbol{v}_0^\delta \|^2_{L^2(B_R)} &= \int_{|\boldsymbol{x}| < R} \left| \int_{\mathbb{R}^2} \phi_\delta(\boldsymbol{x} - \boldsymbol{y}) \boldsymbol{v}_0(\boldsymbol{y}) d\boldsymbol{y} \right|^2 d\boldsymbol{x} \\
&\leq \int_{|\boldsymbol{x}| < R} \left(\int_{\mathbb{R}^2} \phi_\delta(\boldsymbol{x} - \boldsymbol{y}) d\boldsymbol{y} \right) \left( \int_{\mathbb{R}^2} \phi_\delta(\boldsymbol{x} - \boldsymbol{y}) \left| \boldsymbol{v}_0(\boldsymbol{y}) \right|^2  d\boldsymbol{y} \right) d\boldsymbol{x} \\
&\leq \| \boldsymbol{v}_0 \|^2_{L^2(B_{R+1})},
\end{align*}
for any $R >0$. Note that $q_0^\delta = \phi^\delta \ast q_0$ is positive and it follows from $q_0 \in \mathcal{M}(\mathbb{R}^2) \cap H^{-1}_{\mathrm{loc}}(\mathbb{R}^2)$ that $\| q^\delta_0 \|_{L^1} \leq \| q_0 \|_{\mathcal{M}}$ and $\| q^\delta_0 \|_{H^{-1}_{\mathrm{loc}}} \leq \| q_0 \|_{H^{-1}_{\mathrm{loc}}}$. Then, we find that $\boldsymbol{v}_0^\delta$ converges strongly to $\boldsymbol{v}_0$ in $L^2_{\mathrm{loc}}(\mathbb{R}^2)$ and $q^\delta$ converges strongly to $q_0$ in $H^{-1}_{\mathrm{loc}}(\mathbb{R}^2)$.

According to the classical argument for the existence of a smooth solution to the 2D Euler equations \cite{McGrath}, we can obtain a smooth solution of the 2D filtered Euler equations with initial data (\ref{smooth_ID}). Similarly to the 2D Euler equations, smooth solutions of the 2D filtered Euler equations preserve the total vorticity and the second moment of vorticity, which are given by
\begin{align*}
Q(t) &= \int_{\mathbb{R}^2} q(\boldsymbol{x}, t) d\boldsymbol{x}, \\
M(t) &= \int_{\mathbb{R}^2} |\boldsymbol{x}|^2 q(\boldsymbol{x}, t) d\boldsymbol{x},
\end{align*}
respectively. If the initial energy is finite,
\begin{equation*}
E(t) = \int_{\mathbb{R}^2} |\boldsymbol{v}(\boldsymbol{x}, t)|^2 d\boldsymbol{x}
\end{equation*}
is also conserved. However, the smooth initial velocity constructed by (\ref{smooth_ID}) does not belongs to $L^2(\mathbb{R}^2)$ in general. For the case of solutions with infinite energy, we consider the following conserved quantity called the {\it pseudo-energy}.
\begin{equation}
H^{\varepsilon}(t) = H^\varepsilon(q(\boldsymbol{x}, t)) \equiv - \int_{\mathbb{R}^2} \int_{\mathbb{R}^2} G^\varepsilon(\boldsymbol{x} - \boldsymbol{y}) q(\boldsymbol{x}, t) q(\boldsymbol{y}, t) d\boldsymbol{y} d\boldsymbol{x}. \label{pseudo_energy}
\end{equation}
Note that we have $E(t) = H^\varepsilon(t)$ for solutions with globally finite energy. Using conserved quantities, we obtain a decay estimate of the vorticity maximal function for the approximate-solution sequence of the 2D filtered Euler equations.

\begin{lemma}
{\it
Let $q_0^\delta$ be a smooth initial vorticity satisfying
\begin{align*}
q_0^\delta &\geq 0, \\
Q^\delta_0 &\equiv \int_{\mathbb{R}^2} q_0^\delta(\boldsymbol{x}) d\boldsymbol{x} \leq Q_0 < \infty, \\
M^\delta_0 &\equiv \int_{\mathbb{R}^2} \left| \boldsymbol{x} \right|^2 q_0^\delta(\boldsymbol{x}) d\boldsymbol{x} \leq M_0 < \infty, \\
|H^{\varepsilon,\delta}_0| &\equiv |H^\varepsilon(q^\delta_0)| \leq H_0 < \infty.
\end{align*}
Then, there exists a constant $c = c(Q_0, M_0, H_0) > 0$ such that, for any $0 <r \leq 1/4$ and $T > 0$, we have
\begin{equation}
M(r, q^{\varepsilon, \delta}) < c \left[ \log{\left( \frac{1}{2 r + \varepsilon}\right)} \right]^{- 1/2}, \label{VMF_approximate}
\end{equation}
where $q^{\varepsilon, \delta}$ is a solution of the 2D filtered Euler equations with $q^{\varepsilon, \delta}(0) = q^\delta_0$.
} \label{Lemm_VMF}
\end{lemma}
\begin{proof}
Consider the following decomposition of the pseudo-energy,
\begin{align*}
H^{\varepsilon, \delta}(t) &= - \iint_{|\boldsymbol{x} - \boldsymbol{y}| > 1/2} G^\varepsilon(\boldsymbol{x} - \boldsymbol{y}) q^{\varepsilon, \delta}(\boldsymbol{x}, t) q^{\varepsilon, \delta}(\boldsymbol{y}, t) d\boldsymbol{y} d\boldsymbol{x} \\
& \quad - \iint_{|\boldsymbol{x} - \boldsymbol{y}| \leq 1/2} G^\varepsilon(\boldsymbol{x} - \boldsymbol{y}) q^{\varepsilon, \delta}(\boldsymbol{x}, t) q^{\varepsilon, \delta}(\boldsymbol{y}, t) d\boldsymbol{y} d\boldsymbol{x},
\end{align*}
which gives
\begin{align}
&- \iint_{|\boldsymbol{x} - \boldsymbol{y}| \leq 1/2} G^\varepsilon(\boldsymbol{x} - \boldsymbol{y}) q^{\varepsilon, \delta}(\boldsymbol{x}, t) q^{\varepsilon, \delta}(\boldsymbol{y}, t) d\boldsymbol{y} d\boldsymbol{x}  \nonumber \\
&\leq \left| H_0^{\varepsilon, \delta} \right| + \iint_{|\boldsymbol{x} - \boldsymbol{y}| > 1/2} G^\varepsilon(\boldsymbol{x} - \boldsymbol{y}) q^{\varepsilon, \delta}(\boldsymbol{x}, t) q^{\varepsilon, \delta}(\boldsymbol{y}, t) d\boldsymbol{y} d\boldsymbol{x}. \label{est-pseudo_energy}
\end{align}
We estimate the function $G^\varepsilon = G \ast h^\varepsilon$. Note that
\begin{align*}
G^\varepsilon(\boldsymbol{x}) &= \frac{1}{2 \pi} \int_{\mathbb{R}^2} \left( \log{|\boldsymbol{x} - \boldsymbol{y}|} \right) h^\varepsilon(\boldsymbol{y}) d\boldsymbol{y} = \frac{1}{2 \pi} \int_{\mathbb{R}^2} \left( \log{|\boldsymbol{x} - \varepsilon \boldsymbol{y}|} \right) h(\boldsymbol{y}) d\boldsymbol{y} \\
&= \frac{1}{2 \pi} \left[ \int_{|\boldsymbol{y}| \leq 1 } \left( \log{|\boldsymbol{x} - \varepsilon \boldsymbol{y}|} \right) h(\boldsymbol{y}) d\boldsymbol{y} + \int_{|\boldsymbol{y}| > 1 } \left( \log{|\boldsymbol{x} - \varepsilon \boldsymbol{y}|} \right) h(\boldsymbol{y}) d\boldsymbol{y} \right] \\
&\leq \frac{1}{2 \pi} \left[ \log{\left(|\boldsymbol{x}| + \varepsilon \right)} \int_{|\boldsymbol{y}| \leq 1} h(\boldsymbol{y}) d\boldsymbol{y} + \int_{\mathbb{R}^2} \left(|\boldsymbol{x}| + \varepsilon |\boldsymbol{y}| \right) h(\boldsymbol{y}) d\boldsymbol{y} \right].
\end{align*}
Then, for $|\boldsymbol{x}| > 1/2$, it follows that
\begin{align}
G^\varepsilon(\boldsymbol{x}) &\leq c \left[ |\boldsymbol{x}|^2 \int_{\mathbb{R}^2} h(\boldsymbol{y}) d\boldsymbol{y} + \varepsilon \int_{\mathbb{R}^2} |\boldsymbol{y}| h(\boldsymbol{y}) d\boldsymbol{y} \right] \nonumber \\
&\leq c ( |\boldsymbol{x}|^2 + 1 ). \label{est-G^eps-near}
\end{align}
For the case of $|\boldsymbol{x}| \leq 1/2$, since $h$ is positive, it follows that
\begin{align}
G^\varepsilon(\boldsymbol{x}) \leq c \left[ - \log{\left( \frac{1}{|\boldsymbol{x}| + \varepsilon} \right)} + 1 \right], \label{est-G^eps-far}
\end{align}
that is, 
\begin{equation*}
\log{\left( \frac{1}{|\boldsymbol{x}| + \varepsilon} \right)} \leq c \left( - G^\varepsilon(\boldsymbol{x}) + 1 \right).
\end{equation*}
With the conserved quantities, the estimate (\ref{est-G^eps-near}) yields
\begin{align*}
&\iint_{|\boldsymbol{x} - \boldsymbol{y}| > 1/2} G^\varepsilon(\boldsymbol{x} - \boldsymbol{y}) q^{\varepsilon, \delta}(\boldsymbol{x}, t) q^{\varepsilon, \delta}(\boldsymbol{y}, t) d\boldsymbol{y} d\boldsymbol{x} \\
&\leq c \iint_{|\boldsymbol{x} - \boldsymbol{y}| > 1/2} \left( |\boldsymbol{x} - \boldsymbol{y}|^2 + \varepsilon \right) q^{\varepsilon, \delta}(\boldsymbol{x}, t) q^{\varepsilon, \delta}(\boldsymbol{y}, t) d\boldsymbol{y} d\boldsymbol{x} \\
&\leq c \left[ \int_{\mathbb{R}^2} \int_{\mathbb{R}^2} \left( |\boldsymbol{x}|^2 + |\boldsymbol{y}|^2 \right) q^{\varepsilon, \delta}(\boldsymbol{x}, t) q^{\varepsilon, \delta}(\boldsymbol{y}, t) d\boldsymbol{y} d\boldsymbol{x} + \varepsilon \left( Q^\delta_0 \right)^2 \right] \\
&\leq c Q^\delta_0 \left( M^\delta_0 + \varepsilon Q^\delta_0 \right),
\end{align*}
and (\ref{est-G^eps-far}) does
\begin{align*}
&\iint_{|\boldsymbol{x} - \boldsymbol{y}| \leq 1/2} \log{\left( \frac{1}{|\boldsymbol{x} - \boldsymbol{y}| + \varepsilon} \right)} q^{\varepsilon, \delta}(\boldsymbol{x}, t) q^{\varepsilon, \delta}(\boldsymbol{y}, t) d\boldsymbol{y} d\boldsymbol{x} \\
&\leq c \left[ - \iint_{|\boldsymbol{x} - \boldsymbol{y}| \leq 1/2} G^\varepsilon(\boldsymbol{x} - \boldsymbol{y}) q^{\varepsilon, \delta}(\boldsymbol{x}, t) q^{\varepsilon, \delta}(\boldsymbol{y}, t) d\boldsymbol{y} d\boldsymbol{x} + \left( Q^\delta_0 \right)^2 \right].
\end{align*}
Thus, it follows from (\ref{est-pseudo_energy}) that
\begin{equation}
\iint_{|\boldsymbol{x} - \boldsymbol{y}| \leq 1/2} \log{\left( \frac{1}{|\boldsymbol{x} - \boldsymbol{y}| + \varepsilon} \right)} q^{\varepsilon, \delta}(\boldsymbol{x}, t) q^{\varepsilon, \delta}(\boldsymbol{y}, t) d\boldsymbol{y} d\boldsymbol{x}  \leq H_0 + c Q_0 \left( M_0 + Q_0 \right). \label{log-q-q-est}
\end{equation}
Considering the estimate,
\begin{align*}
&\iint_{|\boldsymbol{x} - \boldsymbol{y}| \leq r} \log{\left( \frac{1}{|\boldsymbol{x} - \boldsymbol{y}| + \varepsilon} \right)} q^{\varepsilon, \delta}(\boldsymbol{x}, t) q^{\varepsilon, \delta}(\boldsymbol{y}, t) d\boldsymbol{y} d\boldsymbol{x} \\
&\geq \log{\left( \frac{1}{r + \varepsilon} \right)} \int_{|\boldsymbol{x} - \boldsymbol{x}_0| \leq r/2} \int_{|\boldsymbol{y} - \boldsymbol{x}_0| \leq r/2}  q^{\varepsilon, \delta}(\boldsymbol{x}, t) q^{\varepsilon, \delta}(\boldsymbol{y}, t) d\boldsymbol{y} d\boldsymbol{x}  \\
&\geq \log{\left( \frac{1}{r + \varepsilon} \right)}\left( \int_{|\boldsymbol{x} - \boldsymbol{x}_0| \leq r/2} q^{\varepsilon, \delta}(\boldsymbol{x}, t) d\boldsymbol{x} \right)^2,
\end{align*}
we finally obtain
\begin{align*}
\int_{|\boldsymbol{x} - \boldsymbol{x}_0| < r/2} q^{\varepsilon, \delta}(\boldsymbol{x}, t) d\boldsymbol{x} &\leq c \left( Q_0, M_0, H_0 \right) \left[ \log{\left( \frac{1}{r + \varepsilon} \right)} \right]^{-1/2}
\end{align*}
for $0 < r \leq 1/2$.
\end{proof}

We now confirm that the smoothed initial data (\ref{smooth_ID}) satisfies the assumption of Lemma~\ref{Lemm_VMF}. It easily follows that $Q_0^\delta = \| q_0^\delta \|_{L^1} \leq \| q_0 \|_{\mathcal{M}}$ and $M_0^\delta \leq \overline{R}^2 Q_0$, since $q_0^\delta$ is positive and supp $q_0^\delta \subset B_{\overline{R}}$. As for the pseudo-energy $H_0^{\varepsilon, \delta}$, we have
\begin{equation*}
\left| H_0^{\varepsilon, \delta} \right| = \left| \int_{\mathbb{R}^2} \left( G^\varepsilon \ast q_0^\delta \right)(\boldsymbol{x}) q_0^\delta (\boldsymbol{x}) d\boldsymbol{x} \right| \leq \| G^\varepsilon \ast q_0^\delta \|_{H^1_{\mathrm{loc}}} \| q_0^\delta \|_{H^{-1}_{\mathrm{loc}}}.
\end{equation*}
Owing to $\| q_0^\delta \|_{H^{-1}_{\mathrm{loc}}} \leq \| q_0 \|_{H^{-1}_{\mathrm{loc}}}$, it suffices to show that $\| G^\varepsilon \ast q_0^\delta \|_{H^1_{\mathrm{loc}}}$ is uniformly bounded for $\varepsilon$, $\delta >0$. Considering $G^\varepsilon \ast q_0^\delta = G \ast q_0^{\varepsilon, \delta}$ with $q_0^{\varepsilon, \delta} = h^\varepsilon \ast q_0^\delta$, we have
\begin{align*}
\| G^\varepsilon \ast q_0^\delta \|_{L^2_{\mathrm{loc}}} \leq \| (\chi_{|\boldsymbol{x}| < R_0} G) \ast q_0^{\varepsilon, \delta} \|_{L^2_{\mathrm{loc}}} + \| (\chi_{|\boldsymbol{x}| \geq R_0} G) \ast q_0^{\varepsilon, \delta} \|_{L^2_{\mathrm{loc}}}
\end{align*}
for any $R_0 > 0$. The first term in the right side is estimated as
\begin{equation*}
\| (\chi_{|\boldsymbol{x}| < R_0} G) \ast q_0^{\varepsilon, \delta} \|_{L^2} \leq \| G \|_{L^2(B_{R_0})} \| h^\varepsilon \ast q_0^\delta \|_{L^1} \leq \| G \|_{L^2(B_{R_0})} \| q_0 \|_{\mathcal{M}}.
\end{equation*}
To see the second term, note that 
\begin{equation*}
\| (\chi_{|\boldsymbol{x}| \geq R_0} G) \ast q_0^{\varepsilon, \delta} \|_{L^2(B_R)}^2 \leq  c  \int_{|\boldsymbol{x}| < R} \left( \int_{|\boldsymbol{x} - \boldsymbol{y}| \geq R_0} |\boldsymbol{x} - \boldsymbol{y}| q_0^{\varepsilon, \delta}(\boldsymbol{y}) d\boldsymbol{y} \right)^2 d\boldsymbol{x}
\end{equation*}
for any fixed $R > 0$. Taking $R_0$ such as $R_0 > 2 R$, we have $|\boldsymbol{y}| \geq |\boldsymbol{x} - \boldsymbol{y}| - |\boldsymbol{x}| > R_0 - R > R > |\boldsymbol{x}|$ and 
\begin{align*}
\| (\chi_{|\boldsymbol{x}| \geq R_0} G) \ast q_0^{\varepsilon, \delta} \|_{L^2(B_R)} &\leq c R \int_{|\boldsymbol{y}| \geq R} |\boldsymbol{y}| q_0^{\varepsilon, \delta}(\boldsymbol{y}) d\boldsymbol{y} \\
&\leq c R \int_{|\boldsymbol{y}| \geq R} \int_{|\boldsymbol{z}| < \overline{R}} \left( |\boldsymbol{y} - \boldsymbol{z}| + |\boldsymbol{z}| \right) h^\varepsilon(\boldsymbol{y} -\boldsymbol{z}) q_0^\delta(\boldsymbol{z}) d\boldsymbol{z} d\boldsymbol{y} \\
&\leq c R \left( \| (w_1 h^\varepsilon) \ast q_0^\delta \|_{L^1} + \overline{R} \| h^\varepsilon \ast q_0^\delta \|_{L^1}\right) \\
&\leq c R \left(\varepsilon \| w_1 h \|_{L^1}  + \overline{R} \right) \| q_0 \|_{\mathcal{M}}.
\end{align*}
Thus, we find 
\begin{equation*}
\| G^\varepsilon \ast q_0^\delta \|_{L^2(B_R)} \leq c (R, \overline{R}) \| q_0 \|_{\mathcal{M}}.
\end{equation*}
As for the derivative of $G^\varepsilon \ast q_0^\delta$, we note that the Biot-Savart low gives
\begin{equation*}
| \nabla (G^\varepsilon \ast q_0^\delta) | = | K^\varepsilon \ast q_0^\delta | = | h^\varepsilon \ast (K \ast q_0^\delta ) | = | h^\varepsilon \ast v_0^\delta |. 
\end{equation*}
Then, since it follows that
\begin{align*}
\| \nabla (G^\varepsilon \ast q_0^\delta) \|_{L^2(B_R)} &\leq \| h^\varepsilon \ast (\chi_{|\boldsymbol{x}| < 2 \overline{R}} v_0^\delta) \|_{L^2} + \| h^\varepsilon \ast (\chi_{|\boldsymbol{x}| \geq 2 \overline{R}} v_0^\delta) \|_{L^2(B_R)} \\
&\leq \| v_0^\delta \|_{L^2(B_{2 \overline{R}})} + c(R) \| v_0^\delta \|_{L^\infty(\mathbb{R}^2 \setminus B_{2\overline{R}})}
\end{align*}
and
\begin{equation*}
\| v_0^\delta \|_{L^\infty(\mathbb{R}^2 \setminus B_{2 \overline{R}})} \leq \frac{1}{2 \pi} \sup_{|\boldsymbol{x}| > 2 \overline{R}} \int_{|\boldsymbol{y}| < \overline{R}} \frac{1}{|\boldsymbol{x} - \boldsymbol{y}|} q_0^\delta(\boldsymbol{y}) d\boldsymbol{y} \leq \frac{1}{2\pi \overline{R}} \| q^\delta_0 \|_{L^1},
\end{equation*}
we have
\begin{equation*}
\| \nabla (G^\varepsilon \ast q_0^\delta) \|_{L^2(B_R)} \leq \| v_0 \|_{L^2(B_{2\overline{R}})} + c (R, \overline{R}) \| q_0 \|_{\mathcal{M}}.
\end{equation*}
Combining the above estimates, we obtain 
\begin{equation*}
\left| H_0^{\varepsilon, \delta} \right| \leq c(\overline{R})(\| v_0 \|_{L^2_{\mathrm{loc}}} + \| q_0 \|_{\mathcal{M}}) \| q_0 \|_{H^{-1}_{\mathrm{loc}}},
\end{equation*}
where $c(\overline{R}) > 0$ is independent of $\varepsilon$ and $\delta$.

\subsection{Limiting procedure} \label{limit-procedure}
We consider the $\delta \rightarrow 0$ limit of the approximate-solution sequence. Since we have $\| q^{\varepsilon,\delta}(\cdot, t) \|_{L^1} = \| q^\delta_0 \|_{L^1} \leq \| q_0 \|_{\mathcal{M}}$, there exist a subsequence $\{ q^{\varepsilon,\delta_j} \}_{j\in\mathbb{N}}$ and $q^{\varepsilon} \in L^\infty([0,T];\mathcal{M}(\mathbb{R}^2))$ such that 
\begin{equation}
q^{\varepsilon,\delta_j} \overset{\ast}{\rightharpoonup} q^{\varepsilon} \quad \mathrm{weakly}\ast \ \mathrm{in} \ L^\infty([0,T];\mathcal{M}(\mathbb{R}^2)), \label{lim_q_eps_delta}
\end{equation}
where $\delta_j \rightarrow 0$ as $j \rightarrow \infty$. We now see that $q^{\varepsilon}$ is a weak solution of the 2D filtered Euler equations, that is, $W_L(q^{\varepsilon};\psi) + W^\varepsilon_{NL}(q^{\varepsilon};\psi) = 0$. Since $q^{\varepsilon,\delta_j}$ is a solution of the 2D filtered Euler equations, it is sufficient to show that
\begin{equation*}
W_L(q^{\varepsilon,\delta_j};\psi) \rightarrow W_L(q^{\varepsilon};\psi), \qquad W^\varepsilon_{NL}(q^{\varepsilon,\delta_j};\psi) \rightarrow W^\varepsilon_{NL}(q^{\varepsilon};\psi)
\end{equation*}
as $j \rightarrow \infty$. It is straightforward from (\ref{lim_q_eps_delta}) to check the convergence of the linear term. Before showing the convergence of the nonlinear term, we note that $q^{\varepsilon,\delta_j}$ is uniformly bounded in $\mathrm{Lip}([0,T];H^{-4}(\mathbb{R}^2))$. Indeed, since we have
\begin{align*}
&\left| \int_0^T \int_{\mathbb{R}^2} \partial_t \psi(\boldsymbol{x},t) q^{\varepsilon,\delta_j}(\boldsymbol{x},t) d\boldsymbol{x} dt \right| \\
& \leq \left| \int_0^T \int_{\mathbb{R}^2} \int_{\mathbb{R}^2} H^\varepsilon_{\psi}(\boldsymbol{x},\boldsymbol{y},t) q^{\varepsilon,\delta_j}(\boldsymbol{y},t) q^{\varepsilon,\delta_j}(\boldsymbol{x},t) d\boldsymbol{y} d\boldsymbol{x} dt \right|
\end{align*}
with $\psi\in C_c^\infty(\mathbb{R}^2\times(0,T))$ and $H^\varepsilon_\psi(t) \in C_0(\mathbb{R}^2 \times \mathbb{R}^2)$ satisfies
\begin{align*}
\left| H^\varepsilon_{\psi}(\boldsymbol{x},\boldsymbol{y},t) \right| &\leq \frac{1}{4 \pi} P_K \left( \frac{|\boldsymbol{x} - \boldsymbol{y}|}{\varepsilon} \right) \frac{\left| \nabla \psi(\boldsymbol{x},t) - \nabla \psi(\boldsymbol{y},t) \right|}{|\boldsymbol{x} - \boldsymbol{y}|} \leq \frac{1}{4 \pi} \| \nabla^2 \psi (\cdot, t)\|_{L^\infty}
\end{align*}
for any $t \in [0, T]$, we obtain
\begin{align*}
\left| \int_0^T \int_{\mathbb{R}^2} \partial_t \psi(\boldsymbol{x},t) q^{\varepsilon,\delta_j}(\boldsymbol{x},t) d\boldsymbol{x} dt \right| &\leq c \| q^\delta_0 \|_{\mathcal{M}}^2 \| \psi \|_{L^1([0,T];W^{2, \infty}(\mathbb{R}^2))} \\
& \leq c \| q_0 \|_{\mathcal{M}}^2 \| \psi \|_{L^1([0,T];H^4(\mathbb{R}^2))},
\end{align*}
which implies that $q^{\varepsilon,\delta_j}$ are uniformly bounded in $\mathrm{Lip}([0,T];H^{-4}(\mathbb{R}^2))$ and equicontinuous in time with values in $H^{-4}(\mathbb{R}^2)$. Owing to the equicontinuity and the convergence (\ref{lim_q_eps_delta}), we have 
\begin{equation}
q^{\varepsilon,\delta_j}(\boldsymbol{x},t) q^{\varepsilon,\delta_j}(\boldsymbol{y},t) \overset{\ast}{\rightharpoonup} q^{\varepsilon}(\boldsymbol{x},t) q^{\varepsilon}(\boldsymbol{y},t) \quad \mathrm{weakly}\ast \ \mathrm{in} \ L^\infty([0,T];\mathcal{M}(\mathbb{R}^2 \times \mathbb{R}^2)), \label{conv-qq}
\end{equation}
see Lemma~3.2 in \cite{Schochet}. Thus, since $H^\varepsilon_\psi$ belongs to $C([0,T];(C_0(\mathbb{R}^2))^2)$, we find the convergence of the nonlinear term so that $q^\varepsilon \in C([0,T]; \mathcal{M}(\mathbb{R}^2))$ is a weak solution of the 2D filtered Euler equations with $q^\varepsilon(0) = q_0$. 

In order to show the decay estimate of the maximal vorticity function for $q^{\varepsilon}$, let $\rho \in C_c^\infty[0,\infty)$ be a cut-off function satisfying $0 \leq \rho \leq 1$, $\rho(r) = 1$ for $r \leq 1$ and $\rho(r) = 0$ for $r > 2$. Then, it follows that
\begin{align*}
\sup_{0\leq t \leq T, \boldsymbol{x}_0 \in\mathbb{R}^2 } \int_{|\boldsymbol{x} - \boldsymbol{x}_0| < r} q^{\varepsilon}(\boldsymbol{x},t) d\boldsymbol{x} &\leq \sup_{t , \boldsymbol{x}_0} \int_{\mathbb{R}^2} \rho\left( \frac{|\boldsymbol{x} - \boldsymbol{x}_0|}{r}\right) q^{\varepsilon}(\boldsymbol{x},t) d\boldsymbol{x} \\
& \leq \sup_{t , \boldsymbol{x}_0} \liminf_{j \rightarrow \infty} \int_{\mathbb{R}^2} \rho\left( \frac{|\boldsymbol{x} - \boldsymbol{x}_0|}{r}\right) q^{\varepsilon, \delta_j}(\boldsymbol{x},t) d\boldsymbol{x} \\
& \leq \sup_{t , \boldsymbol{x}_0} \liminf_{j \rightarrow \infty} \int_{|\boldsymbol{x} - \boldsymbol{x}_0| < 2 r}  q^{\varepsilon, \delta_j}(\boldsymbol{x},t) d\boldsymbol{x}
\end{align*}
and thus the estimate (\ref{VMF_approximate}) yields
\begin{equation}
M(r, q^{\varepsilon}) < c \left[ \log{\left( \frac{1}{4r + \varepsilon}\right)} \right]^{- 1/2} \label{VMF_filtered}
\end{equation}
for $0 < r \leq 1/8$. Especially, (\ref{log-q-q-est}) and (\ref{conv-qq}) yield
\begin{equation}
\iint_{|\boldsymbol{x} - \boldsymbol{y}| \leq r} \log{\left( \frac{1}{|\boldsymbol{x} - \boldsymbol{y}| + \varepsilon} \right)} q^{\varepsilon}(\boldsymbol{x}, t) q^{\varepsilon}(\boldsymbol{y}, t) d\boldsymbol{y} d\boldsymbol{x} \leq c  \label{est_int_log_q}
\end{equation}
for $0 < r \leq 1/4$. The estimate (\ref{est_int_log_q}) implies that the filtered velocity field defined by $\boldsymbol{u}^\varepsilon = \boldsymbol{K}^\varepsilon \ast q^\varepsilon$ is uniformly bounded in $L^\infty([0,T];L^2_{\mathrm{loc}}(\mathbb{R}^2))$.

\begin{lemma}
{\it
For any $R > 0$, there exists a constant $c= c(R) > 0$ independent of $\varepsilon$ such that
\begin{equation*}
\sup_{0\leq t \leq T}\int_{|\boldsymbol{x}| \leq R} \left| \boldsymbol{u}^\varepsilon(\boldsymbol{x}, t) \right|^2 d\boldsymbol{x} \leq c.
\end{equation*}
} \label{bdd-u-eps}
\end{lemma}

\begin{proof}
Using the filtered Biot-Savart formula, that is, $\boldsymbol{u}^\varepsilon = \boldsymbol{K}^\varepsilon \ast q^\varepsilon$, we find 
\begin{align*}
\int_{|\boldsymbol{x}| \leq R} \left| \boldsymbol{u}^\varepsilon(\boldsymbol{x}, t) \right|^2 d\boldsymbol{x} =  \int_{\mathbb{R}^2} \int_{\mathbb{R}^2} \mathscr{K}^\varepsilon(\boldsymbol{y}, \boldsymbol{z},R) q^\varepsilon(\boldsymbol{y}, t) q^\varepsilon(\boldsymbol{z}, t)  d\boldsymbol{y} d\boldsymbol{z},
\end{align*}
where
\begin{equation*}
\mathscr{K}^\varepsilon(\boldsymbol{y}, \boldsymbol{z},R) = \int_{|\boldsymbol{x}| \leq R} \boldsymbol{K}^\varepsilon(\boldsymbol{x} - \boldsymbol{y}) \cdot \boldsymbol{K}^\varepsilon(\boldsymbol{x} - \boldsymbol{z}) d\boldsymbol{x}.
\end{equation*}
Consider the following decomposition with the cut-off function $\rho$.
\begin{align*}
\int_{|\boldsymbol{x}| \leq R} \left| \boldsymbol{u}^\varepsilon(\boldsymbol{x}, t) \right|^2 d\boldsymbol{x} &=  \iint \rho\left(\frac{|\boldsymbol{y} - \boldsymbol{z}|}{R_0} \right) \mathscr{K}^\varepsilon(\boldsymbol{y}, \boldsymbol{z},R) q^\varepsilon(\boldsymbol{y}, t) q^\varepsilon(\boldsymbol{z}, t)  d\boldsymbol{y} d\boldsymbol{z} \\
&\quad + \iint \left[ 1 - \rho\left(\frac{|\boldsymbol{y} - \boldsymbol{z}|}{R_0} \right) \right] \mathscr{K}^\varepsilon(\boldsymbol{y}, \boldsymbol{z},R) q^\varepsilon(\boldsymbol{y}, t) q^\varepsilon(\boldsymbol{z}, t)  d\boldsymbol{y} d\boldsymbol{z},
\end{align*}
where $R_0 > 0$. According to Lemma~2.2 in \cite{Liu}, it follows that
\begin{equation*}
\left| \mathscr{K}^\varepsilon(\boldsymbol{y}, \boldsymbol{z},R) \right| \leq c \log{\left(\frac{1}{|\boldsymbol{y} - \boldsymbol{z}| + \varepsilon} \right)}
\end{equation*}
for $|\boldsymbol{y} - \boldsymbol{z}| \leq 2 R_0$. Hence, owing to (\ref{est_int_log_q}), we obtain
\begin{equation*}
\left| \iint \rho\left(\frac{|\boldsymbol{y} - \boldsymbol{z}|}{R_0} \right) \mathscr{K}^\varepsilon(\boldsymbol{y}, \boldsymbol{z},R) q^\varepsilon(\boldsymbol{y}, t) q^\varepsilon(\boldsymbol{z}, t)  d\boldsymbol{y} d\boldsymbol{z} \right| \leq c.
\end{equation*}
For the case of $|\boldsymbol{y} - \boldsymbol{z}| > R_0$, it is easily confirmed that $\mathscr{K}^\varepsilon(R) \in C_0(\mathbb{R}^2 \times \mathbb{R}^2)$ and there exists a constant $c = c(R) > 0$ such that $\left| \mathscr{K}^\varepsilon(\boldsymbol{y}, \boldsymbol{z},R) \right| \leq c$, which yields  
\begin{equation*}
\left| \iint \left[ 1 - \rho\left(\frac{|\boldsymbol{y} - \boldsymbol{z}|}{R_0} \right) \right]  \mathscr{K}^\varepsilon(\boldsymbol{y}, \boldsymbol{z},R) q^\varepsilon(\boldsymbol{y}, t) q^\varepsilon(\boldsymbol{z}, t)  d\boldsymbol{y} d\boldsymbol{z} \right| \leq c \| q_0 \|^2_{\mathcal{M}}.
\end{equation*}
\end{proof}

\section{Convergence to the Euler equations}
\subsection{Convergence of the vorticity}
We consider the $\varepsilon \rightarrow 0$ limit of solutions of the 2D filtered Euler equations. Owing to $\| q^\varepsilon (\cdot, t)\|_{\mathcal{M}} = \| q_0 \|_{\mathcal{M}}$, there exists a subsequence $\{ q^{\varepsilon_j} \}$ and $q \in L^\infty([0,T]; \mathcal{M}(\mathbb{R}^2))$ such that
\begin{equation}
q^{\varepsilon_j} \overset{\ast}{\rightharpoonup} q  \quad \mathrm{weakly}\ast \ \mathrm{in} \ L^\infty([0,T];\mathcal{M}(\mathbb{R}^2))  \label{lim_q_eps}
\end{equation}
as $j \rightarrow \infty$. Then, the filtered vorticity $\omega^{\varepsilon_j}$ also converges to $q$ weakly--$\ast$ in $L^\infty([0,T];\mathcal{M}(\mathbb{R}^2))$. Indeed, $q^\varepsilon$ and $\omega^\varepsilon$ are uniformly bounded in $C([0,T];\mathcal{M}(\mathbb{R}^2))$ and we have 
\begin{align*}
\left| \int_{\mathbb{R}^2} \psi(\boldsymbol{x})\left( \omega^\varepsilon(\boldsymbol{x},t) - q^\varepsilon(\boldsymbol{x},t) \right) d\boldsymbol{x} \right| &= \left| \int_{\mathbb{R}^2} \left( (h^\varepsilon \ast \psi) (\boldsymbol{x}) - \psi(\boldsymbol{x}) \right) q^\varepsilon(\boldsymbol{x},t) d\boldsymbol{x} \right| \\
&\leq \| \left( h^\varepsilon \ast \psi \right) -\psi \|_{L^\infty} \| q_0 \|_{\mathcal{M}}.
\end{align*}
for any $\psi \in C^1_0(\mathbb{R}^2)$. Since it follows that
\begin{align*}
\left| \left( h^\varepsilon \ast \psi \right)(\boldsymbol{x}) - \psi(\boldsymbol{x}) \right| &\leq  \int_{\mathbb{R}^2} h(\boldsymbol{y}) \left| \psi(\boldsymbol{x} - \varepsilon\boldsymbol{y}) - \psi(\boldsymbol{x}) \right| d\boldsymbol{y} \\
&\leq \varepsilon \| \nabla \psi \|_{L^\infty} \int_{\mathbb{R}^2} |\boldsymbol{y}| \left|h(\boldsymbol{y})\right| d\boldsymbol{y},
\end{align*}
we obtain
\begin{equation*}
\left| \int_{\mathbb{R}^2} \psi(\boldsymbol{x})\left( \omega^\varepsilon(\boldsymbol{x},t) - q^\varepsilon(\boldsymbol{x},t) \right) d\boldsymbol{x} \right| \leq \varepsilon \| \nabla \psi \|_{L^\infty} \| w_1 h \|_{L^1} \| q_0 \|_{\mathcal{M}}.
\end{equation*}
By a density argument, $\omega^\varepsilon - q^\varepsilon$ converges to zero weakly-$\ast$ in $L^\infty([0,T];\mathcal{M}(\mathbb{R}^2))$ in the $\varepsilon \rightarrow 0$ limit, which gives the desired result.

We now show that the limit vorticity $q$ in (\ref{lim_q_eps}) is a solution of the 2D Euler equations in the sense of Definition~\ref{def-sol-vorticity}. Note that $q$ has a decay estimate for its vorticity maximal function, that is,
\begin{equation}
M(r, q) < c \left[ \log{\left( \frac{1}{8r + \varepsilon}\right)} \right]^{- 1/2} \label{VMF}
\end{equation}
for sufficiently small $r > 0$, which is shown in the same manner as (\ref{VMF_filtered}). Recall the weak form of the 2D Euler equations: $W_L(q;\psi) + W_{NL}(q;\psi) = 0$, where
\begin{align*}
W_L(q;\psi) &= \int_0^T \int_{\mathbb{R}^2} \partial_t \psi(\boldsymbol{x},t) q(\boldsymbol{x},t) d\boldsymbol{x} dt, \\
W_{NL}(q;\psi) &= \int_0^T \int_{\mathbb{R}^2} \int_{\mathbb{R}^2} H_{\psi}(\boldsymbol{x},\boldsymbol{y},t) q(\boldsymbol{y},t) q(\boldsymbol{x},t) d\boldsymbol{y} d\boldsymbol{x} dt
\end{align*}
with $\psi \in C_c^\infty(\mathbb{R}^2 \times (0,T))$, and the 2D filtered Euler equations: $W_L(q^\varepsilon;\psi) + W_{NL}^\varepsilon(q^\varepsilon;\psi) = 0$, where
\begin{align*}
W_L(q^\varepsilon;\psi) &= \int_0^T \int_{\mathbb{R}^2} \partial_t \psi(\boldsymbol{x},t) q^\varepsilon(\boldsymbol{x},t) d\boldsymbol{x} dt, \\
W_{NL}^\varepsilon(q^\varepsilon;\psi) &= \int_0^T \int_{\mathbb{R}^2} \int_{\mathbb{R}^2} H^\varepsilon_{\psi}(\boldsymbol{x},\boldsymbol{y},t) q^\varepsilon(\boldsymbol{y},t) q^\varepsilon(\boldsymbol{x},t) d\boldsymbol{y} d\boldsymbol{x} dt.
\end{align*}
The convergence of the linear term, $W_L(q^{\varepsilon_j};\psi) \rightarrow W_L(q;\psi)$, directly follows from (\ref{lim_q_eps}). To see the convergence of the nonlinear term, we divide $W_{NL}^\varepsilon(q^\varepsilon;\psi) - W_{NL}(q;\psi)$ into two parts in the same way as shown in \cite{Bardos}:
\begin{align*}
&W_{NL}^\varepsilon(q^\varepsilon;\psi) - W_{NL}(q;\psi) \\
&= \int_0^T \iint \left( H_\psi^\varepsilon(\boldsymbol{x},\boldsymbol{y},t) - H_\psi(\boldsymbol{x},\boldsymbol{y},t) \right) q^\varepsilon(\boldsymbol{y},t) q^\varepsilon(\boldsymbol{x},t) d\boldsymbol{y} d\boldsymbol{x} dt \\
&\quad + \int_0^T \iint H_\psi(\boldsymbol{x},\boldsymbol{y},t) \left( q^\varepsilon(\boldsymbol{y},t) q^\varepsilon(\boldsymbol{x},t) - q(\boldsymbol{y},t) q(\boldsymbol{x},t) \right) d\boldsymbol{y} d\boldsymbol{x} dt \\
&= I_1 + I_2.
\end{align*}
Consider the following decomposition of $I_1$ with the cut-off function $\rho$.
\begin{align*}
I_1 &= \int_0^T \iint \left[ 1 - \rho\left( \frac{|\boldsymbol{x} - \boldsymbol{y}|}{\varepsilon R} \right) \right] \left( H_\psi^\varepsilon(\boldsymbol{x},\boldsymbol{y},t) - H_\psi(\boldsymbol{x},\boldsymbol{y},t) \right) q^\varepsilon(\boldsymbol{y},t) q^\varepsilon(\boldsymbol{x},t) d\boldsymbol{y} d\boldsymbol{x} dt \\
&\quad + \int_0^T \iint \rho\left( \frac{|\boldsymbol{x} - \boldsymbol{y}|}{\varepsilon R} \right) \left( H_\psi^\varepsilon(\boldsymbol{x},\boldsymbol{y},t) - H_\psi(\boldsymbol{x},\boldsymbol{y},t) \right) q^\varepsilon(\boldsymbol{y},t) q^\varepsilon(\boldsymbol{x},t) d\boldsymbol{y} d\boldsymbol{x} dt \\
&= I_{11} + I_{12}.
\end{align*}
From the definitions of $H_\psi$ and $H_\psi^\varepsilon$, it follows that
\begin{equation*}
\left| H_\psi^\varepsilon(\boldsymbol{x},\boldsymbol{y},t) - H_\psi(\boldsymbol{x},\boldsymbol{y},t) \right| \leq \frac{1}{4 \pi} \left[ 1 - P_K\left(\frac{|\boldsymbol{x} - \boldsymbol{y}|}{\varepsilon}\right) \right] \| \nabla^2 \psi \|_{L^\infty}.
\end{equation*}
Since $P_K(r)$ converges to $1$ as $r \rightarrow \infty$, for any $\delta > 0$, there exists $R_\delta > 0$ such that $1 - P_K(r) < \delta$ for any $r > R_\delta$. Thus, setting $\zeta_R(\boldsymbol{x}) \equiv (1 - \rho(|\boldsymbol{x}|/R))(1 - P_K(|\boldsymbol{x}|))$ and taking $R > R_\delta$, we find $\zeta_R \in C_0(\mathbb{R}^2)$ with $\| \zeta_R \|_{L^\infty} \leq \delta$ and 
\begin{align*}
|I_{11}| &\leq c \| \nabla^2 \psi \|_{L^\infty} \int_0^T \iint \zeta_R \left( \frac{|\boldsymbol{x} - \boldsymbol{y}|}{\varepsilon} \right) q^\varepsilon(\boldsymbol{y},t) q^\varepsilon(\boldsymbol{x},t) d\boldsymbol{y} d\boldsymbol{x} dt \\
&\leq c \| \nabla^2 \psi \|_{L^\infty} \| \zeta_R \|_{L^\infty} \int_0^T \| q^\varepsilon(\cdot,t) \|_{\mathcal{M}}^2 dt \\
&\leq c \delta T \| \nabla^2 \psi \|_{L^\infty} \| q_0 \|_{\mathcal{M}}^2. 
\end{align*}
We estimate $I_{12}$ with the vorticity maximal function.
\begin{align*}
|I_{12}| &\leq c \| \nabla^2 \psi \|_{L^\infty} \int_0^T \iint \rho\left( \frac{|\boldsymbol{x} - \boldsymbol{y}|}{\varepsilon R} \right) q^\varepsilon(\boldsymbol{y},t) q^\varepsilon(\boldsymbol{x},t) d\boldsymbol{y} d\boldsymbol{x} dt \\
&\leq c T \| \nabla^2 \psi \|_{L^\infty} \| q_0 \|_{\mathcal{M}} M(2 \varepsilon R, q^\varepsilon)  \\
&\leq c T \| \nabla^2 \psi \|_{L^\infty} \| q_0 \|_{\mathcal{M}} \left[ \log{\left( \frac{1}{4\varepsilon R + \varepsilon} \right)} \right]^{-1/2},
\end{align*}
for any sufficiently small $\varepsilon > 0$. Hence, there exists $\varepsilon_0 = \varepsilon_0(\delta) > 0$ such that, for any $\varepsilon < \varepsilon_0$, we have 
\begin{equation*}
|I_1| \leq c \delta T \| \nabla^2 \psi \|_{L^\infty} \| q_0 \|_{\mathcal{M}} (1 +  \| q_0 \|_{\mathcal{M}}).
\end{equation*}
To estimate $I_2$, similarly to $I_1$, we consider the following decomposition with $\rho$.
\begin{align*}
I_2 &= \int_0^T \!\!\! \iint \left[ 1 - \rho\left( \frac{|\boldsymbol{x} - \boldsymbol{y}|}{\delta} \right) \right] H_\psi(\boldsymbol{x},\boldsymbol{y},t) \left( q^\varepsilon(\boldsymbol{y},t) q^\varepsilon(\boldsymbol{x},t) - q(\boldsymbol{y},t) q(\boldsymbol{x},t) \right) d\boldsymbol{y} d\boldsymbol{x} dt \\
&\quad + \int_0^T \!\!\! \iint \rho\left( \frac{|\boldsymbol{x} - \boldsymbol{y}|}{\delta} \right) H_\psi(\boldsymbol{x},\boldsymbol{y},t) \left( q^\varepsilon(\boldsymbol{y},t) q^\varepsilon(\boldsymbol{x},t) - q(\boldsymbol{y},t) q(\boldsymbol{x},t) \right) d\boldsymbol{y} d\boldsymbol{x} dt \\
&= I_{21} + I_{22}.
\end{align*}
Since we find $q^\varepsilon \in \mathrm{Lip}([0,T];H^{-4}(\mathbb{R}^2))$ and 
\begin{equation*}
q^{\varepsilon_j}(\boldsymbol{y},t) q^{\varepsilon_j}(\boldsymbol{x},t) \overset{\ast}{\rightharpoonup}q(\boldsymbol{y},t) q(\boldsymbol{x},t) \quad \mathrm{weakly}\ast \ \mathrm{in} \ L^\infty([0,T];\mathcal{M}(\mathbb{R}^2 \times \mathbb{R}^2))
\end{equation*}
in the same manner as $q^{\varepsilon,\delta_j}$ in Section~\ref{limit-procedure}, it follows from 
\begin{equation*}
\left[ 1 - \rho\left( \frac{|\boldsymbol{x} - \boldsymbol{y}|}{\delta} \right) \right] H_\psi(\boldsymbol{x},\boldsymbol{y},t) \in C([0,T];(C_0(\mathbb{R}^2))^2)
\end{equation*}
that $I_{21} \rightarrow 0$ as $\varepsilon \rightarrow 0$. As for the convergence of $I_{22}$, we use the vorticity maximal functions.
\begin{align*}
|I_{22}| &= \| H_\psi \|_{L^\infty(\mathbb{R}^2\times[0,T])} \int_0^T \left[ \iint \rho\left( \frac{|\boldsymbol{x} - \boldsymbol{y}|}{\delta} \right) q^\varepsilon(\boldsymbol{y},t) q^\varepsilon(\boldsymbol{x},t) d\boldsymbol{y} d\boldsymbol{x} dt \right.\\
&\quad + \left. \iint \rho\left( \frac{|\boldsymbol{x} - \boldsymbol{y}|}{\delta} \right) q(\boldsymbol{y},t) q(\boldsymbol{x},t) d\boldsymbol{y} d\boldsymbol{x} \right] dt \\
&\leq T \| H_\psi \|_{L^\infty} \| q_0 \|_{\mathcal{M}} \left( M(2 \delta, q^\varepsilon ) + M(2 \delta, q ) \right) \\
&\leq c T \| H_\psi \|_{L^\infty} \| q_0 \|_{\mathcal{M}} \left[ \log{\left( \frac{1}{4\delta + \varepsilon} \right)} \right]^{-1/2}.
\end{align*}
Hence, we have $I_{22} \rightarrow 0$ so that $I_{2} \rightarrow 0$ as $\delta$, $\varepsilon \rightarrow 0$. Consequently, we obtain $W_{NL}(q^{\varepsilon_j}; \psi) \rightarrow W_{NL}(q; \psi)$ as $j \rightarrow \infty$.

In order to show the Lipschitz continuity of $q$, we recall that $q^\varepsilon$ is uniformly bounded in $L^\infty([0,T];\mathcal{M}(\mathbb{R}^2))$ and $\mathrm{Lip}([0,T];H^{-4}(\mathbb{R}^2))$. Since $\mathcal{M}(\mathbb{R}^2)$ is continuously embedded in $H^{-s}_{\mathrm{loc}}(\mathbb{R}^2)$ and $H^{-s}_{\mathrm{loc}}(\mathbb{R}^2)$ is compactly embedded in $H^{-4}(\mathbb{R}^2)$ for $1 < s < 4$, the compactness theorem gives the existence of a limit $\overline{q}$ such that a subsequence $\{ q^{\varepsilon_j}\}$ converges to $\overline{q}$ in $\mathrm{Lip}([0,T];H_{\mathrm{loc}}^{-4}(\mathbb{R}^2))$. From a density argument, we find that $q = \overline{q}$ in $L^\infty([0,T];\mathcal{M}(\mathbb{R}^2)) \cap \mathrm{Lip}([0,T];H_{\mathrm{loc}}^{-4}(\mathbb{R}^2))$. Therefore, we conclude that $q$ is a weak solution of the 2D Euler equations in the sense of Definition~\ref{def-sol-vorticity}.

\subsection{Convergence of the velocity field}

We show that the filtered velocity $\boldsymbol{u}^\varepsilon$ converges to a weak solution of the 2D Euler equations. The proof proceeds in a similar fashion to that in \cite{Liu}, though their method is for an discretized problem with the vortex method. Note that $q^\varepsilon \in C([0, T];\mathcal{M}(\mathbb{R}^2))$ satisfies
\begin{equation*}
\int_0^T \int_{\mathbb{R}^2} \left( \partial_t \psi  + \nabla \psi \cdot \boldsymbol{u}^\varepsilon  \right) q^\varepsilon d\boldsymbol{x} dt = 0
\end{equation*}
for any $\psi\in C_c^\infty(\mathbb{R}^2\times(0,T))$. This is equivalent to 
\begin{equation}
\int_0^T \int_{\mathbb{R}^2} \left( \nabla^\perp\partial_t \psi \cdot \boldsymbol{v}^\varepsilon + \nabla^\perp \otimes \nabla \psi : \boldsymbol{u}^\varepsilon \otimes \boldsymbol{u}^\varepsilon \right) d\boldsymbol{x} dt = \int_0^T W_1(t) dt, \label{weak-form-velo-1}
\end{equation}
where $W_1$ is given by
\begin{equation}
W_1 = \int_{\mathbb{R}^2} \nabla \psi \cdot \boldsymbol{u}^\varepsilon \left( \omega^\varepsilon - q^\varepsilon \right)  d\boldsymbol{x}. \label{W1}
\end{equation}

First, we show that $\int_0^T W_1(t) dt$ converges to zero as $\varepsilon \rightarrow 0$. It follows that
\begin{align*}
W_1 &= \int_{\mathbb{R}^2}\int_{\mathbb{R}^2} (\nabla \psi(\boldsymbol{x})\cdot \boldsymbol{u}^\varepsilon(\boldsymbol{x}) - \nabla \psi(\boldsymbol{y}) \cdot \boldsymbol{u}^\varepsilon(\boldsymbol{y}) ) h^\varepsilon(\boldsymbol{x} - \boldsymbol{y}) q^\varepsilon(\boldsymbol{y}) d\boldsymbol{y} d\boldsymbol{x} \\
&= \iint (\nabla \psi(\boldsymbol{x}) - \nabla \psi(\boldsymbol{y}) ) \cdot \boldsymbol{u}^\varepsilon(\boldsymbol{x}) h^\varepsilon(\boldsymbol{x} - \boldsymbol{y}) q^\varepsilon(\boldsymbol{y}) d\boldsymbol{y} d\boldsymbol{x} \\
& \quad + \iint \nabla \psi(\boldsymbol{y})\cdot (\boldsymbol{u}^\varepsilon(\boldsymbol{x}) - \boldsymbol{u}^\varepsilon(\boldsymbol{y}) ) h^\varepsilon(\boldsymbol{x} - \boldsymbol{y}) q^\varepsilon(\boldsymbol{y}) d\boldsymbol{y} d\boldsymbol{x} \\
&\equiv W_{11} + W_{12}.
\end{align*}
Note that $|\varepsilon \boldsymbol{K}^\varepsilon(\boldsymbol{x})| \leq c$ for any $\boldsymbol{x} \in \mathbb{R}^2$ and $|\varepsilon \boldsymbol{K}^\varepsilon(\boldsymbol{x})| \leq c \varepsilon^{1/2}$ for $|\boldsymbol{x}| \geq \varepsilon^{1/2}$. Then, we have
\begin{align*}
|W_{11}| &\leq \|\nabla^2 \psi \|_{L^\infty} \iiint |\boldsymbol{x} - \boldsymbol{y}| |\boldsymbol{K}^\varepsilon(\boldsymbol{x} - \boldsymbol{z})| h^\varepsilon(\boldsymbol{x} - \boldsymbol{y}) q^\varepsilon(\boldsymbol{z}) q^\varepsilon(\boldsymbol{y}) d\boldsymbol{z} d\boldsymbol{y} d\boldsymbol{x} \\
&\leq c \|\nabla^2 \psi \|_{L^\infty} \left[\iiint \rho\left(\frac{|\boldsymbol{x} - \boldsymbol{z}|}{\varepsilon^{1/2}} \right) \frac{|\boldsymbol{x} - \boldsymbol{y}|}{\varepsilon} h^\varepsilon(\boldsymbol{x} - \boldsymbol{y}) q^\varepsilon(\boldsymbol{z}) q^\varepsilon(\boldsymbol{y}) d\boldsymbol{z} d\boldsymbol{y} d\boldsymbol{x} \right.\\
&\quad \left. + \varepsilon^{1/2} \iiint \left[ 1 - \rho\left(\frac{|\boldsymbol{x} - \boldsymbol{z}|}{\varepsilon^{1/2}} \right) \right] \frac{|\boldsymbol{x} - \boldsymbol{y}|}{\varepsilon} h^\varepsilon(\boldsymbol{x} - \boldsymbol{y}) q^\varepsilon(\boldsymbol{z}) q^\varepsilon(\boldsymbol{y}) d\boldsymbol{z} d\boldsymbol{y} d\boldsymbol{x} \right] \\
&\equiv c \|\nabla^2 \psi \|_{L^\infty} \left( W_{111} + W_{112} \right).
\end{align*}
To estimate $W_{111}$, we consider
\begin{align*}
W_{111} &= \iiint \rho\left(\frac{|\boldsymbol{x} - \boldsymbol{z}|}{\varepsilon^{1/2}} \right) \rho\left(\frac{|\boldsymbol{y} - \boldsymbol{z}|}{\varepsilon^{1/3}} \right) \frac{|\boldsymbol{x} - \boldsymbol{y}|}{\varepsilon} h^\varepsilon(\boldsymbol{x} - \boldsymbol{y}) q^\varepsilon(\boldsymbol{z}) q^\varepsilon(\boldsymbol{y}) d\boldsymbol{z} d\boldsymbol{y} d\boldsymbol{x} \\
& + \iiint \rho\left(\frac{|\boldsymbol{x} - \boldsymbol{z}|}{\varepsilon^{1/2}} \right) \left[ 1 - \rho\left(\frac{|\boldsymbol{y} - \boldsymbol{z}|}{\varepsilon^{1/3}} \right) \right] \frac{|\boldsymbol{x} - \boldsymbol{y}|}{\varepsilon} h^\varepsilon(\boldsymbol{x} - \boldsymbol{y}) q^\varepsilon(\boldsymbol{z}) q^\varepsilon(\boldsymbol{y}) d\boldsymbol{z} d\boldsymbol{y} d\boldsymbol{x}.
\end{align*}
The first integral is estimated with the vorticity maximal function.
\begin{align*}
&\iiint \rho\left(\frac{|\boldsymbol{x} - \boldsymbol{z}|}{\varepsilon^{1/2}} \right) \rho\left(\frac{|\boldsymbol{y} - \boldsymbol{z}|}{\varepsilon^{1/3}} \right) \frac{|\boldsymbol{x} - \boldsymbol{y}|}{\varepsilon} h^\varepsilon(\boldsymbol{x} - \boldsymbol{y}) q^\varepsilon(\boldsymbol{z}) q^\varepsilon(\boldsymbol{y}) d\boldsymbol{z} d\boldsymbol{y} d\boldsymbol{x} \\
&\leq \| w_1 h \|_{L^1} \iint \rho\left(\frac{|\boldsymbol{y} - \boldsymbol{z}|}{\varepsilon^{1/3}} \right) q^\varepsilon(\boldsymbol{z}) q^\varepsilon(\boldsymbol{y}) d\boldsymbol{z} d\boldsymbol{y} \\
&\leq \| w_1 h \|_{L^1} \| q_0 \|_{\mathcal{M}} M(2\varepsilon^{1/3}, q^\varepsilon) \\
&\leq c \| w_1 h \|_{L^1} \| q_0 \|_{\mathcal{M}} \left[ \log{\left(\frac{1}{4\varepsilon^{1/3} + \varepsilon}\right)}\right]^{-1/2}.
\end{align*}
Owing to $w_3 h \in L^\infty(\mathbb{R}^2)$, the second one is estimated as
\begin{align*}
&\iiint \rho\left(\frac{|\boldsymbol{x} - \boldsymbol{z}|}{\varepsilon^{1/2}} \right) \left[ 1 - \rho\left(\frac{|\boldsymbol{y} - \boldsymbol{z}|}{\varepsilon^{1/3}} \right) \right] \frac{|\boldsymbol{x} - \boldsymbol{y}|}{\varepsilon} h^\varepsilon(\boldsymbol{x} - \boldsymbol{y}) q^\varepsilon(\boldsymbol{z}) q^\varepsilon(\boldsymbol{y}) d\boldsymbol{z} d\boldsymbol{y} d\boldsymbol{x} \\
&\leq \| w_3 h \|_{L^\infty} \iiint \rho\left(\frac{|\boldsymbol{x} - \boldsymbol{z}|}{\varepsilon^{1/2}} \right) \left[ 1 - \rho\left(\frac{2|\boldsymbol{x} - \boldsymbol{y}|}{\varepsilon^{1/3}} \right) \right] \frac{1}{|\boldsymbol{x} - \boldsymbol{y}|^2} q^\varepsilon(\boldsymbol{z}) q^\varepsilon(\boldsymbol{y}) d\boldsymbol{z} d\boldsymbol{y} d\boldsymbol{x} \\
&\leq c \varepsilon^{-2/3} \| w_3 h \|_{L^\infty} \| q_0 \|_{\mathcal{M}}^2 \int_{\mathbb{R}^2} \rho\left(\frac{|\boldsymbol{x}|}{\varepsilon^{1/2}} \right) d\boldsymbol{x} \\
&\leq c \varepsilon^{1/3} \| w_3 h \|_{L^\infty} \| q_0 \|_{\mathcal{M}}^2.
\end{align*}
Thus, $W_{111} \rightarrow 0$ as $\varepsilon \rightarrow 0$. Since it is easily checked that
\begin{align*}
W_{112} \leq \varepsilon^{1/2} \| w_1 h \|_{L^1} \| q_0 \|_{\mathcal{M}}^2,
\end{align*}
we obtain $W_{11} \rightarrow 0$ as $\varepsilon \rightarrow 0$. As for the convergence of $W_{12}$, we have
\begin{align*}
W_{12}&= \int_{\mathbb{R}^2}\int_{\mathbb{R}^2} \nabla \psi(\boldsymbol{y})\cdot (\boldsymbol{u}^\varepsilon(\boldsymbol{x}) - \boldsymbol{u}^\varepsilon(\boldsymbol{y}) ) h^\varepsilon(\boldsymbol{x} - \boldsymbol{y}) q^\varepsilon(\boldsymbol{y}) d\boldsymbol{y} d\boldsymbol{x} \\
&= \iint \nabla \psi(\boldsymbol{y}) \cdot ( h^\varepsilon \ast (\boldsymbol{K}^\varepsilon - \boldsymbol{K})) (\boldsymbol{y} - \boldsymbol{z})  q^\varepsilon(\boldsymbol{z}) q^\varepsilon(\boldsymbol{y}) d\boldsymbol{z} d\boldsymbol{y} \\
&= \frac{1}{2} \iint (\nabla \psi(\boldsymbol{y}) - \nabla \psi(\boldsymbol{z})) \cdot ( h^\varepsilon \ast (\boldsymbol{K}^\varepsilon - \boldsymbol{K})) (\boldsymbol{y} - \boldsymbol{z}) q^\varepsilon(\boldsymbol{z}) q^\varepsilon(\boldsymbol{y}) d\boldsymbol{z} d\boldsymbol{y}
\end{align*}
Note that it follows that
\begin{equation*}
|( h^\varepsilon \ast (\boldsymbol{K}^\varepsilon - \boldsymbol{K})) (\boldsymbol{x})| \leq c \frac{\varepsilon}{|\boldsymbol{x}| + \varepsilon} \frac{1}{|\boldsymbol{x}|}.
\end{equation*}
Then, we find 
\begin{align*}
|W_{12}|&\leq c \| \nabla^2 \psi \|_{L^\infty} \iint \rho\left(\frac{|\boldsymbol{y} - \boldsymbol{z}|}{\varepsilon^{1/3}} \right) q^\varepsilon(\boldsymbol{z}) q^\varepsilon(\boldsymbol{y}) d\boldsymbol{z} d\boldsymbol{y} \\
& \quad + c \| \nabla \psi \|_{L^\infty} \iint \left[ 1 - \rho\left(\frac{|\boldsymbol{y} - \boldsymbol{z}|}{\varepsilon^{1/3}} \right) \right] \frac{\varepsilon}{|\boldsymbol{y} - \boldsymbol{z}|^2} q^\varepsilon(\boldsymbol{z}) q^\varepsilon(\boldsymbol{y}) d\boldsymbol{z} d\boldsymbol{y} \\
&\leq c \| q_0 \|_{\mathcal{M}} \left[ \| \nabla^2 \psi \|_{L^\infty} \left[ \log{\left(\frac{1}{4\varepsilon^{1/3} + \varepsilon}\right)}\right]^{-1/2} + \varepsilon^{1/3} \| \nabla \psi \|_{L^\infty} \| q_0 \|_{\mathcal{M}} \right]
\end{align*}
so that $W_{12} \rightarrow 0$ as $\varepsilon \rightarrow 0$. Consequently, we obtain $\int_0^T W_1(t) dt \rightarrow 0$ as $\varepsilon \rightarrow 0$.

It is important to remark that it follows from the above proof that 
\begin{equation}
\left| \int_0^T W_1(t) dt \right| \leq c \| \nabla^\perp \psi \|_{L^1([0,T]; W^{1, \infty}(\mathbb{R}^2))}. \label{W1-bdd}
\end{equation}
Thus, combining (\ref{W1-bdd}) and Lemma~\ref{bdd-u-eps}, we find that (\ref{weak-form-velo-1}) yields
\begin{align*}
\left| \int_0^T \int_{\mathbb{R}^2} \nabla^\perp \psi \cdot \partial_t \boldsymbol{v}^\varepsilon d\boldsymbol{x} dt \right| &\leq c \| \nabla^\perp \psi \|_{L^1([0,T]; W^{1, \infty}(\mathbb{R}^2))} \\
&\leq c \| \nabla^\perp \psi \|_{L^1([0,T]; H^3(\mathbb{R}^2))}
\end{align*}
and thus
\begin{equation}
\left| \int_0^T \int_{\mathbb{R}^2} \nabla^\perp \psi \cdot \partial_t \boldsymbol{u}^\varepsilon d\boldsymbol{x} dt \right| \leq c \| \nabla^\perp \psi \|_{L^1([0,T]; H^3(\mathbb{R}^2))}. \label{Lip-u-eps}
\end{equation}
Hence, $\boldsymbol{v}^\varepsilon$ and $\boldsymbol{u}^\varepsilon$ are uniformly bounded in $\mathrm{Lip}([0,T];H_{\mathrm{loc}}^{-3}(\mathbb{R}^2))$.

Next, we see that the filtered velocity $\boldsymbol{u}^\varepsilon$ converges to a weak solution of the 2D Euler equations in the sense of Definition~\ref{def-sol-velocity}. Note that Lemma~\ref{bdd-u-eps} implies that there exist a subsequence $\{\boldsymbol{u}^{\varepsilon_j} \}$ and $\boldsymbol{u} \in L^2_{\mathrm{loc}}(\mathbb{R}^2 \times [0,T])$ such that
\begin{equation}
\boldsymbol{u}^{\varepsilon_j} \rightharpoonup \boldsymbol{u} \quad \mathrm{weakly} \ \mathrm{in} \ L^2_{\mathrm{loc}}(\mathbb{R}^2 \times [0,T]). \label{conv-u-eps}
\end{equation}
Then, it follows from (\ref{Lip-u-eps}) that $\boldsymbol{u}$ belongs to $\mathrm{Lip}([0,T]; H_{\mathrm{loc}}^{-3}(\mathbb{R}^2))$. On the other hand, the weak form (\ref{weak-form-velo-1})
is rewritten by
\begin{equation*}
\int_0^T \int_{\mathbb{R}^2} \left( \nabla^\perp\partial_t \psi \cdot \boldsymbol{u}^\varepsilon + \nabla^\perp \otimes \nabla \psi : \boldsymbol{u}^\varepsilon \otimes \boldsymbol{u}^\varepsilon \right) d\boldsymbol{x} dt = \int_0^T W_1(t) dt + \int_0^T W_2(t) dt,
\end{equation*}
where $W_1$ is given in (\ref{W1}) and $W_2$ is 
\begin{align*}
W_2 = \int_{\mathbb{R}^2}  \nabla^\perp\partial_t \psi \cdot \left( \boldsymbol{u}^\varepsilon - \boldsymbol{v}^\varepsilon \right)  d\boldsymbol{x}.
\end{align*}
We show that $\int_0^T W_2 dt$ converges to zero in the $\varepsilon \rightarrow 0$ limit. It is sufficient to show that
\begin{equation}
\int_0^T \int_{|\boldsymbol{x}| \leq R} | \boldsymbol{u}^\varepsilon - \boldsymbol{v}^\varepsilon | d\boldsymbol{x} dt \rightarrow 0, \label{conv-u-eps-v-eps}
\end{equation}
as $\varepsilon \rightarrow 0$. The integral in (\ref{conv-u-eps-v-eps}) is decomposed into the following three parts with the approximate-solution sequences, $\{ \boldsymbol{v}^{\varepsilon, \delta} \}$ and $\{ \boldsymbol{u}^{\varepsilon, \delta} \}$.
\begin{align*}
\int_{|\boldsymbol{x}| \leq R} | \boldsymbol{u}^\varepsilon - \boldsymbol{v}^\varepsilon | d\boldsymbol{x} & \leq \int_{|\boldsymbol{x}| \leq R} | \boldsymbol{u}^\varepsilon - \boldsymbol{u}^{\varepsilon, \delta} | d\boldsymbol{x} \\
&\quad + \int_{|\boldsymbol{x}| \leq R} | \boldsymbol{u}^{\varepsilon, \delta} -  \boldsymbol{v}^{\varepsilon, \delta} | d\boldsymbol{x} + \int_{|\boldsymbol{x}| \leq R} | \boldsymbol{v}^{\varepsilon, \delta} - \boldsymbol{v}^\varepsilon | d\boldsymbol{x} \\
& \equiv J_1 + J_2 + J_3.
\end{align*}
For the convergence of $J_1$, recall that $\boldsymbol{u}^{\varepsilon}$ and $\boldsymbol{u}^{\varepsilon, \delta}$ are continuous in $\mathbb{R}^2 \times [0,T]$ and, owing to (\ref{lim_q_eps_delta}),
\begin{equation*}
(\boldsymbol{u}^\varepsilon - \boldsymbol{u}^{\varepsilon, \delta})(\boldsymbol{x}) = \int_{\mathbb{R}^2} \boldsymbol{K}^\varepsilon(\boldsymbol{x} - \boldsymbol{y}) (q^{\varepsilon}(\boldsymbol{y}) - q^{\varepsilon, \delta}(\boldsymbol{y}) ) d\boldsymbol{y} 
\end{equation*}
converges pointwise to zero as $\delta \rightarrow 0$ for any fixed $\varepsilon > 0$. Since we have $\| \boldsymbol{u}^\varepsilon \|_{L^\infty} \leq \| \boldsymbol{K}^\varepsilon \|_{L^\infty} \| q_0 \|_{\mathcal{M}}$ and $\| \boldsymbol{u}^{\varepsilon, \delta} \|_{L^\infty} \leq \| \boldsymbol{K}^\varepsilon \|_{L^\infty} \| q_0 \|_{\mathcal{M}}$, the dominated convergence theorem yields $J_1 \rightarrow 0$ as $\delta \rightarrow 0$. As for $J_2$, we use Proposition~3.8 in \cite{G.}, which is the estimate for the $L^1$-norm of $\boldsymbol{K}^{\varepsilon} -  \boldsymbol{K}$. Then, it follows that
\begin{align*}
J_2 \leq \int_{\mathbb{R}^2} \left| (\boldsymbol{K}^\varepsilon - \boldsymbol{K}) \ast q^{\varepsilon, \delta} \right| d\boldsymbol{x} \leq \| \boldsymbol{K}^{\varepsilon} -  \boldsymbol{K} \|_{L^1} \| q^{\varepsilon, \delta} \|_{L^1} \leq \varepsilon \| w_1 h \|_{L^1} \| q_0 \|_{\mathcal{M}}
\end{align*}
and thus the convergence of $J_2$ holds. To see the convergence of $J_3$, the following lemma plays an important role, see \cite{Diperna,Majda-1} for its proof.

\begin{lemma}
{\it
Let $\{ \boldsymbol{v}_n \}_{n \in \mathbb{N}}$ be a sequence in $L^\infty([0,T]; L^2_{\mathrm{loc}}(\mathbb{R}^2))$ satisfying the following conditions:
\begin{itemize}
\item[(i)] $\operatorname{div} \boldsymbol{v}_n = 0$ in the sense of distributions.
\item[(ii)] $q_n = \operatorname{curl} \boldsymbol{v}_n$ satisfies $\sup_{0\leq t \leq T} \| q_n(\cdot, t) \|_{L^1} \leq c(T)$.
\item[(iii)] $\{ \boldsymbol{v}_n \}$ are uniformly bounded in $\mathrm{Lip}([0,T];H_{\mathrm{loc}}^{-L}(\mathbb{R}^2))$ for some $L > 0$.
\end{itemize}
Then, there exist a subsequence $\{ \boldsymbol{v}_{n_j} \}$ and $\boldsymbol{v} \in L^\infty([0,T]; L^2_{\mathrm{loc}}(\mathbb{R}^2))$ with $\operatorname{div} \boldsymbol{v} = 0$ in the sense of distributions such that 
\begin{equation*}
\int_0^T \int_{|\boldsymbol{x}| \leq R} | \boldsymbol{v}_{n_j} - \boldsymbol{v}| d\boldsymbol{x} dt \rightarrow 0 
\end{equation*}
and $q_{n_j} \overset{\ast}{\rightharpoonup} q = \operatorname{curl} \boldsymbol{v}$ weakly--$\ast$ in $L^\infty([0,T];\mathcal{M}(\mathbb{R}^2))$ as $j \rightarrow \infty$.
} \label{L1-conv-v-eps}
\end{lemma}

Since $\{ \boldsymbol{v}^{\varepsilon, \delta} \}$ satisfies the assumption in Lemma~\ref{L1-conv-v-eps}, we have the existence of the $\delta \rightarrow 0$ limit velocity field $\overline{\boldsymbol{v}}^\varepsilon$ and 
\begin{equation*}
\int_0^T \int_{|\boldsymbol{x}| \leq R} | \boldsymbol{v}^{\varepsilon, \delta} - \overline{\boldsymbol{v}}^\varepsilon | d\boldsymbol{x} dt \rightarrow 0.
\end{equation*}
Considering the convergence of the vorticity (\ref{lim_q_eps_delta}), we find that $\overline{\boldsymbol{v}}^\varepsilon = \boldsymbol{v}^\varepsilon$ in the sense of distributions and thus  $\int_0^T J_3 dt \rightarrow 0$ as $\delta \rightarrow 0$. Hence, taking the limit of $\varepsilon \rightarrow 0$ after the $\delta \rightarrow 0$ limit, we obtain (\ref{conv-u-eps-v-eps}).

Combining the convergences of $W_1$ and $W_2$, we have the following weak consistency.
\begin{equation*}
\int_0^T \int_{\mathbb{R}^2} \left( \nabla^\perp\partial_t \psi \cdot \boldsymbol{u}^\varepsilon + \nabla^\perp \otimes \nabla \psi : \boldsymbol{u}^\varepsilon \otimes \boldsymbol{u}^\varepsilon \right) d\boldsymbol{x} dt \rightarrow 0 \quad \mathrm{as} \quad \varepsilon \rightarrow 0.
\end{equation*}
Since (\ref{conv-u-eps}) yields 
\begin{equation*}
\int_0^T \int_{\mathbb{R}^2} \nabla^\perp\partial_t \psi \cdot \boldsymbol{u}^\varepsilon  d\boldsymbol{x} dt \rightarrow \int_0^T \int_{\mathbb{R}^2} \nabla^\perp\partial_t \psi \cdot \boldsymbol{u}  d\boldsymbol{x} dt
\end{equation*}
and the convergence,
\begin{equation*}
\int_0^T \int_{\mathbb{R}^2} \nabla^\perp \otimes \nabla \psi : \boldsymbol{u}^\varepsilon \otimes \boldsymbol{u}^\varepsilon  d\boldsymbol{x} dt \rightarrow \int_0^T \int_{\mathbb{R}^2} \nabla^\perp \otimes \nabla \psi : \boldsymbol{u} \otimes \boldsymbol{u} d\boldsymbol{x} dt,
\end{equation*}
can be shown by the Delort's idea in the same manner as the proof in \cite{Liu}, the limit velocity field $\boldsymbol{u} \in L^\infty([0,T]; L^2_{\mathrm{loc}}(\mathbb{R}^2)) \cap \mathrm{Lip}([0,T]; H_{\mathrm{loc}}^{-3}(\mathbb{R}^2))$ is the desired weak solution of the 2D Euler equations in the sense of Definition~\ref{def-sol-velocity}.

\subsection{Convergence of the vortex method}

We prove Theorem~\ref{convergence-thm2}. In the vortex method, initial vorticity (\ref{init-pv}) belongs to $\mathcal{M}(\mathbb{R}^2)$, but not $H^{-1}_{\mathrm{loc}}(\mathbb{R}^2)$. This is the only difference from Theorem~\ref{convergence-thm1} in which initial vorticity is in $\mathcal{M}(\mathbb{R}^2) \cap H^{-1}_{\mathrm{loc}}(\mathbb{R}^2)$. According to the proof of Theorem~\ref{convergence-thm1}, we utilize the assumption $q_0 \in H^{-1}_{\mathrm{loc}}(\mathbb{R}^2)$ only for deriving the vorticity maximal function. In other words, if we obtain the uniform decay of the vorticity maximal function for the approximate solution (\ref{sol-pv}), the convergence theorem can be proven in the same way as Theorem~\ref{convergence-thm1}. Note that the vorticity maximal function in the vortex method is defined by
\begin{equation*}
\sup_{0\leq t \leq T, \boldsymbol{x}_0 \in\mathbb{R}^2 } \sum_{|\boldsymbol{x}_n(t) - \boldsymbol{x}_0| < r} \Gamma_n,
\end{equation*}
where $\boldsymbol{x}_n(t)$ is a solution to the Hamiltonian dynamical system (\ref{FPV}) with the Hamiltonian,
\begin{equation*}
\mathscr{H}^\varepsilon(t) = - \sum_{m, n \in \mathbb{N}} \Gamma_m \Gamma_n G^\varepsilon \left( \boldsymbol{x}^\varepsilon_m(t) - \boldsymbol{x}^\varepsilon_n(t) \right),
\end{equation*}
which is a conserved quantity, i.e., $\mathscr{H}^\varepsilon(t) = \mathscr{H}^\varepsilon(0)$ for any $t > 0$. 

The proof of the uniform decay proceeds in the same manner as that in \cite{Liu}. We begin with the following lemma.

\begin{lemma}
{\it
Suppose that there exits a constant $c > 0$ such that $\eta / \varepsilon \leq c$. Then,  $\mathscr{H}^\varepsilon(0)$ is bounded independent of $\varepsilon$ and $\eta$.
} \label{bdd_disc_H}
\end{lemma}

\begin{proof}
We recall the pseudo-energy (\ref{pseudo_energy}),
\begin{equation*}
H^\varepsilon_0 \equiv H^\varepsilon(q_0) = - \int_{\mathbb{R}^2} \int_{\mathbb{R}^2} G^\varepsilon(\boldsymbol{x} - \boldsymbol{y}) q_0(\boldsymbol{x}) q_0(\boldsymbol{y}) d\boldsymbol{x} d\boldsymbol{y},
\end{equation*}
and that, for $q_0 \in \mathcal{M}(\mathbb{R}^2) \cap H^{-1}_{\mathrm{loc}}(\mathbb{R}^2)$, $H^\varepsilon_0$ is bounded independent of $\varepsilon$. Then, we have 
\begin{align*}
\left| \mathscr{H}^\varepsilon(0) - H^\varepsilon_0 \right| \leq \sum_{m, n \in \mathbb{N}} \int_{S_m} \int_{S_n} \left|G^\varepsilon(\boldsymbol{c}_m - \boldsymbol{c}_n) - G^\varepsilon(\boldsymbol{x} - \boldsymbol{y}) \right| q_0(\boldsymbol{x}) q_0(\boldsymbol{y}) d\boldsymbol{x} d\boldsymbol{y}.
\end{align*}
Note that, for any $(\boldsymbol{x}, \boldsymbol{y}) \in S_m \times S_n$, it follows that
\begin{align*}
\left|G^\varepsilon(\boldsymbol{c}_m - \boldsymbol{c}_n) - G^\varepsilon(\boldsymbol{x} - \boldsymbol{y}) \right| \leq \| \nabla G^\varepsilon \|_{L^\infty} \left| (\boldsymbol{x} -\boldsymbol{c}_m) - (\boldsymbol{y} -\boldsymbol{c}_n) \right| \leq c \frac{\eta}{\varepsilon},
\end{align*}
since $\| \nabla G^\varepsilon \|_{L^\infty} = \| \boldsymbol{K}^\varepsilon \|_{L^\infty} \leq c /\varepsilon$. Thus, owing the assumption, we obtain
\begin{equation*}
\left| \mathscr{H}^\varepsilon(0) - H^\varepsilon_0 \right| \leq c \sum_{m, n \in \mathbb{N}} \int_{S_m} \int_{S_n}  q_0(\boldsymbol{x}) q_0(\boldsymbol{y}) d\boldsymbol{x} d\boldsymbol{y} = c \| q_0 \|^2_{\mathcal{M}}
\end{equation*}
so that $|\mathscr{H}^\varepsilon(0)| \leq |H^\varepsilon_0| + | \mathscr{H}^\varepsilon(0) - H^\varepsilon_0 | \leq c$, where $c > 0$ is independent of $\varepsilon$ and $\eta$.
\end{proof}

\begin{lemma}
{\it
There exists a constant $c > 0$ such that, for any $0 < r \leq 1/4$ and $T > 0$, we have
\begin{equation}
\sup_{0\leq t \leq T, \boldsymbol{x}_0 \in\mathbb{R}^2 } \sum_{|\boldsymbol{x}_n - \boldsymbol{x}_0| < r} \Gamma_n < c \left[ \log{\left( \frac{1}{2 r + \varepsilon}\right)} \right]^{- 1/2}. \label{DVMF_approximate}
\end{equation}
} \label{Lemm_DVMF}
\end{lemma}

\begin{proof}
From the definition of $\mathscr{H}^\varepsilon$, we have
\begin{align*}
&- \sum_{|\boldsymbol{x}^\varepsilon_m - \boldsymbol{x}^\varepsilon_n| \leq 1/2 } \Gamma_m \Gamma_n G^\varepsilon \left( \boldsymbol{x}^\varepsilon_m(t) - \boldsymbol{x}^\varepsilon_n(t) \right) \\
&\leq \left| \mathscr{H}(t) \right| + \sum_{|\boldsymbol{x}^\varepsilon_m - \boldsymbol{x}^\varepsilon_n| > 1/2 } \Gamma_m \Gamma_n G^\varepsilon \left( \boldsymbol{x}^\varepsilon_m(t) - \boldsymbol{x}^\varepsilon_n(t) \right).
\end{align*}
It follows from (\ref{est-G^eps-near}) and (\ref{est-G^eps-far}) that
\begin{align*}
& c \sum_{|\boldsymbol{x}^\varepsilon_m - \boldsymbol{x}^\varepsilon_n| \leq 1/2 } \Gamma_m \Gamma_n \log{\left( \frac{1}{|\boldsymbol{x}^\varepsilon_m(t) - \boldsymbol{x}^\varepsilon_n(t)| + \varepsilon} \right)} \\
&\leq \left| \mathscr{H}(0) \right| + c \sum_{m, n \in \mathbb{N}} \Gamma_m \Gamma_n \left( |\boldsymbol{x}^\varepsilon_m(t)|^2 + |\boldsymbol{x}^\varepsilon_n(t)|^2 + 1 \right).
\end{align*}
Owing to Lemma~\ref{bdd_disc_H} and the invariance of the second moment for point vortices, i.e., $\sum_{n\in \mathbb{N}} \Gamma_n |\boldsymbol{x}^\varepsilon_n(t)|^2 = \sum_{n\in \mathbb{N}} \Gamma_n |\boldsymbol{x}_n(0)|^2$, we find 
\begin{equation*}
\sum_{|\boldsymbol{x}^\varepsilon_m - \boldsymbol{x}^\varepsilon_n| \leq 1/2 } \Gamma_m \Gamma_n \log{\left(\frac{1}{|\boldsymbol{x}^\varepsilon_m(t) - \boldsymbol{x}^\varepsilon_n(t)| + \varepsilon} \right)} \leq c,
\end{equation*}
Since it follows that
\begin{align*}
&\sum_{|\boldsymbol{x}^\varepsilon_m - \boldsymbol{x}^\varepsilon_n| \leq 1/2 } \Gamma_m \Gamma_n \log{\left( \frac{1}{|\boldsymbol{x}^\varepsilon_m(t) - \boldsymbol{x}^\varepsilon_n(t)| + \varepsilon} \right)} \geq \log{\left( \frac{1}{r + \varepsilon} \right)} \left( \sum_{|\boldsymbol{x}^\varepsilon_n - \boldsymbol{x}_0| \leq r/2} \Gamma_n \right)^2
\end{align*}
for any $ 0 < r \leq 1/2$, we obtain the desired result.
\end{proof}



\end{document}